\documentclass[11pt,a4paper]{amsart}
\usepackage[a4paper,margin=1.1in]{geometry}
\usepackage{amssymb,mathtools}
\usepackage{enumitem}
\usepackage{cite}
\usepackage{booktabs}
\usepackage[expansion=false]{microtype}
\usepackage{hyperref}
\allowdisplaybreaks
\hypersetup{colorlinks=true,linkcolor=blue,urlcolor=blue}
\theoremstyle{plain}
\newtheorem{theorem}{Theorem}[section]
\newtheorem{proposition}[theorem]{Proposition}
\newtheorem{lemma}[theorem]{Lemma}
\newtheorem{corollary}[theorem]{Corollary}
\theoremstyle{definition}
\newtheorem{definition}[theorem]{Definition}
\newtheorem{assumption}[theorem]{Assumption}
\theoremstyle{remark}
\newtheorem{remark}[theorem]{Remark}
\newtheorem{example}[theorem]{Example}

\numberwithin{equation}{section}

\newcommand{\R}{\mathbb{R}}
\newcommand{\eps}{\varepsilon}
\newcommand{\Lcal}{\mathcal{L}}
\newcommand{\Hp}{\mathcal{H}_{p,h}}
\newcommand{\norm}[1]{\left\lVert #1\right\rVert}
\newcommand{\dd}{\,\mathrm{d}}
\newcommand{\supp}{\operatorname{supp}}
\newcommand{\one}{\mathbf{1}}
\newcommand{\CC}{d_X}
\newcommand{\BX}{B_X}

\newenvironment{inwords}{\begin{quote}\small\textbf{In words.}\ }{\end{quote}}
\newenvironment{idea}{\begin{quote}\small\textbf{Idea.}\ }{\end{quote}}
\newenvironment{roadmap}{\begin{quote}\small\textbf{Plan of the proof.}\ }{\end{quote}}

\begin{document}

\title[Sharp Lifespan Estimates]{Sharp Lifespan Estimates for Heat Equations with H\"ormander Operators}

\author{Michael Ruzhansky}
\author{Berikbol T. Torebek$^{*}$}
\thanks{$^{*}$Corresponding author.}

\address{Michael Ruzhansky\\
Department of Mathematics: Analysis, Logic and Discrete Mathematics\\
Ghent University, Belgium\\
and School of Mathematical Sciences\\
Queen Mary University of London, United Kingdom}
\email{michael.ruzhansky@ugent.be}

\address{Berikbol T. Torebek\\
Institute of Mathematics and Mathematical Modeling\\
28 Shevchenko str., 050010 Almaty, Kazakhstan\\
and SDU University, Kaskelen, Kazakhstan}
\email{torebek@math.kz}


\begin{abstract}
In this paper, we study the Cauchy problem $u_t-\Lcal u=h(t)u^p$,
$u(0)=\eps u_0\ge0$, where $p>1$, $h\ge0$ is only locally integrable, and
$\Lcal=\sum_{j=1}^mX_j^2$ is a homogeneous H\"ormander sum of squares on
$\R^n$ of homogeneous dimension $q$. Diffusion spreads the solution out and
weakens the nonlinearity, while the reaction $h(t)u^p$ may force blow-up. We
prove that the outcome is decided by the single quantity
$\Hp(T)=\int_0^Th(t)(1+t)^{-\frac q2(p-1)}\,\dd t$, in which the reaction at
time $t$ is weighted by the factor by which diffusion has weakened the
solution: small data are global if $\Hp(\infty)<\infty$, while otherwise the
lifespan $T_\eps$ satisfies
$c\,\eps^{-(p-1)}\le\Hp(T_\eps)\le C\,\eps^{-(p-1)}$ for small $\eps$. Our
technique of proof is based on a backward-semigroup functional combined with
a Gaussian-envelope energy and an ODE cap telescoped along interlaced
geometric grids. The result
contains the Euclidean, Heisenberg, Carnot and Grushin cases, yields explicit lifespan
laws for power, logarithmic, exponential, periodic and intermittent
coefficients, and exhibits effects with no autonomous counterpart, such as a
lifespan with no power law in $\eps$. We also prove regularity and a priori
estimates for classical solutions and discuss what remains true for a general
nonlinearity $f(u)$.
\end{abstract}

\maketitle

\section{Introduction}\label{sec:intro}

In this paper, we study the Cauchy problem for the semilinear heat equation
with a time-dependent reaction
\begin{equation}\label{eq:main}
 \begin{cases}
 u_t-\Lcal u=h(t)u^p, & (t,x)\in(0,T)\times\R^n,\\
 u(0,x)=\eps u_0(x), & x\in\R^n,
 \end{cases}
 \qquad \Lcal=\sum_{j=1}^mX_j^2,
\end{equation}
where $p>1$, $\eps>0$ measures the size of the data, and
\begin{equation}\label{eq:h-ass}
 h\in L^1_{\rm loc}([0,\infty)),\quad h\ge0,
 \qquad
 0\le u_0\in L^1(\R^n)\cap L^\infty(\R^n),\quad u_0\not\equiv0 .
\end{equation}

Two effects compete in \eqref{eq:main}. The diffusion $\Lcal$ spreads the
solution over larger and larger regions and thereby lowers its height, which
weakens the nonlinear term $u^p$. The reaction $h(t)u^p$ works in the opposite
direction: it feeds on the solution and, if it acts for long enough, forces it
to blow up in finite time. The \emph{lifespan} $T_\eps\in(0,\infty]$ is the
largest time up to which the solution exists. The question of this paper is:
\emph{how does $T_\eps$ depend on $\eps$ and on the coefficient $h$ when $\eps$
is small?}

The operator $\Lcal$ is built from smooth vector fields $X_1,\dots,X_m$ on $\R^n$
that are homogeneous of degree one with respect to non-isotropic dilations
\begin{equation}\label{eq:dil}
 \delta_\lambda(x)=(\lambda^{\sigma_1}x_1,\dots,\lambda^{\sigma_n}x_n),
 \qquad 1=\sigma_1\le\dots\le\sigma_n,
\end{equation}
and that satisfy H\"ormander's rank condition (see Assumption~\ref{ass:H}
below). The number $q=\sigma_1+\dots+\sigma_n$ is the \emph{homogeneous
dimension}; it replaces the dimension $n$ in every scaling law (for
$\Lcal=\Delta$ one has $q=n$). No group law on $\R^n$ is assumed.

\begin{assumption}[Homogeneous H\"ormander structure]\label{ass:H}
$X_1,\dots,X_m$ are linearly independent smooth real vector fields on $\R^n$
such that \textup{(H1)} each $X_j$ is homogeneous of degree one with respect
to \eqref{eq:dil}, i.e.\ $X_j(f\circ\delta_\lambda)=\lambda(X_jf)\circ\delta_\lambda$;
\textup{(H2)} $\operatorname{rank}\mathrm{Lie}\{X\}(0)=n$.
\end{assumption}

In words, (H1) says that $\Lcal$ has the same scaling structure as the
Laplacian, with $x_i$ scaled by $\lambda^{\sigma_i}$ and time by $\lambda^2$;
this is what makes the heat kernel of $\Lcal$ behave like $t^{-q/2}$ on the
diagonal. (H2) is the rank condition, which makes $\Lcal$ hypoelliptic even
though it need not be elliptic. Examples are the Laplacian, the sub-Laplacians
of the Heisenberg group and of all Carnot groups, and the Grushin and Martinet
operators (Subsection~\ref{subsec:examples}).

Let $T_\eps\in(0,\infty]$ be the maximal time of existence of the mild
solution (Definition~\ref{def:mild}; a mild solution is a solution of the
Duhamel formula \eqref{eq:mild}, in which the solution is the linear heat flow
of the data plus the heat flow of the nonlinear term). Throughout, we write
\[
 \alpha:=\tfrac q2(p-1),\qquad
 \Hp(T):=\int_0^Th(t)(1+t)^{-\alpha}\,\dd t,\quad T\in[0,\infty].
\]
We shall see that $\Hp$ is the natural ``clock'' of the problem;
Subsection~\ref{subsec:heuristic} explains why.

\smallskip
\noindent\textbf{Notation.} $\norm\cdot_k$ is the norm of $L^k(\R^n)$,
$1\le k\le\infty$; $P_t=e^{t\Lcal}$ is the heat semigroup and $\Gamma_t(x,y)$ its
kernel; $\CC$ is the Carnot--Carath\'eodory distance of the fields
$X_1,\dots,X_m$; $f\asymp g$ means $c\le f/g\le C$; and $c,C$ denote positive
constants that depend only on $X$, $p$ and $u_0$ unless stated otherwise.

\subsection{Historical background}

The literature on problems of type \eqref{eq:main} falls into three groups.

\subsubsection{The classical heat equation}
For $h\equiv1$ and $\Lcal=\Delta$ the problem goes back to Fujita
\cite{Fujita1966}, who showed that $p_F=1+2/n$ separates blow-up of all
nontrivial solutions from the existence of small global solutions; the
critical case belongs to Hayakawa \cite{Hayakawa1973}, Sugitani
\cite{Sugitani1975} and Kobayashi--Sirao--Tanaka
\cite{KobayashiSiraoTanaka1977}, and the lifespan was studied by Lee and Ni
\cite{LeeNi1992} (see also \cite{Weissler1981,Levine1990,GuiWang1995,
IkedaSobajima2019,QuittnerSouplet2019}).

\subsubsection{Time-dependent reactions}
For $h(t)\asymp(1+t)^\beta$ the threshold becomes $1+2(1+\beta)/n$
\cite{Meier1988,Meier1990}, and blow-up and global-existence criteria for more
general $h$ and $f(u)$ are given in \cite{LoayzaPaixao2014,MeglioliPunzo2026},
and for coupled systems with time-weighted sources in
\cite{CastilloLoayza2019}.

\subsubsection{Heisenberg group, Grushin and H\"ormander operators}
On the Heisenberg group the Fujita exponent was found by Zhang
\cite{Zhang1998} (see also \cite{Oka, Jleli1, Jleli2, Pascucci, PohozaevVeron2000}) and lifespan estimates for local solutions were obtained by
Georgiev--Palmieri \cite{GeorgievPalmieri2021}; see also
\cite{BorikhanovRuzhanskyTorebek2022} for the equation with a forcing term.
Ruzhansky--Yessirkegenov \cite{RuzhanskyYessirkegenov2022} studied existence and
non-existence of global solutions on sub-Riemannian manifolds and found the
Fujita exponent on connected unimodular Lie groups. For Grushin operators,
local well-posedness and global existence for small data were obtained in
\cite{OliveiraViana2024,Kogoj1}, while nonexistence of nontrivial global
solutions for $p\le1+2/q$ is contained in
\cite{ChatzakouKassymovRuzhansky2026} (see also
\cite{DAmbrosioLucente2003,FinoViana2026a}); lifespan estimates were obtained
in \cite{FinoKirane2026}, and problems with mixed local and nonlocal
nonlinearities were studied in \cite{FinoViana2026a}. The global heat-kernel theory of homogeneous
H\"ormander operators
\cite{BiagiBonfiglioli2017,BiagiBonfiglioli2024,BiagiBonfiglioliBramanti2022,
BiagiBramanti2023} allowed Chatzakou, Kassymov and Ruzhansky
\cite{ChatzakouKassymovRuzhansky2024,ChatzakouKassymovRuzhansky2026} to find
the Fujita exponent for homogeneous H\"ormander sums of squares and to give
necessary conditions for blow-up (and sufficient conditions for global
existence) for $\varphi(t)f(u)$; for $f(u)=u^p$ these involve the integral
$\int_1^\infty\varphi(t)t^{-\alpha}\dd t$.

\subsection{Motivation}

The results above are qualitative in $h$: they give conditions for blow-up or
for global existence, but not the lifespan $T_\eps$ as a function of the
amplitude $\eps$ of the data for a general coefficient $h$. This leads us to
the following questions.

\smallskip
\noindent\textbf{Question 1.} Is there a single quantity, built from $h$, $p$
and $q$, that determines the lifespan $T_\eps$ of small solutions from above
and from below, for \emph{every} $h\ge0$ in $L^1_{\rm loc}$, without any
regularity, monotonicity or doubling condition on $h$?

\smallskip
\noindent\textbf{Question 2.} Can the same quantity decide whether small data
produce global solutions, and what is the critical amplitude of the data?

\smallskip
\noindent\textbf{Question 3.} What do the resulting laws look like for explicit
coefficients (power, logarithmic, exponential, periodic, intermittent), which
effects have no autonomous counterpart, what can be said about classical
solutions (regularity and a priori estimates), and what remains true for a
general nonlinearity $f(u)$?

\smallskip
\noindent The main purpose of this paper is to answer these questions. We show
that the integral $\Hp$ above is, quantitatively, the exact clock of the
lifespan.

\subsection{Heuristic: why \texorpdfstring{$\Hp$}{H} is the clock}\label{subsec:heuristic}

Before stating the theorems we explain, without proof, where the quantity
$\Hp$ comes from. Switch the reaction off first. Then $u(t)=\eps P_tu_0$ has
total mass $\eps\norm{u_0}_1$ for all $t$, and by the decay estimate
\eqref{eq:L1Linfty} its height is about $\eps t^{-q/2}$. Mass is height times
volume, so the solution occupies a region of volume about $t^{q/2}$.

Now switch the reaction on, and ask how much mass it creates per unit time at
time $s$:
\[
 h(s)\int u(s)^p\dd x\ \approx\ h(s)\,\bigl(\eps s^{-q/2}\bigr)^{p}\,s^{q/2}
 =\eps^{p}\,h(s)\,s^{-\frac q2(p-1)}=\eps^p\,h(s)\,s^{-\alpha}.
\]
Compared with the initial mass, which is of order $\eps$, the \emph{relative}
mass created up to time $T$ is therefore of order
\[
 \eps^{p-1}\int_0^T h(s)\,s^{-\alpha}\dd s\ \approx\ \eps^{p-1}\,\Hp(T)
\]
(the shift $1+t$ in the definition of $\Hp$ only removes the singularity of
$s^{-\alpha}$ at $s=0$). This suggests the following picture.
\begin{itemize}[nosep,leftmargin=1.6em]
\item If $\eps^{p-1}\Hp(T)\ll1$, the reaction is a small perturbation of the
heat flow up to time $T$: the solution exists and decays like the linear flow.
\item If $\eps^{p-1}\Hp(T)\gtrsim1$, the reaction has created as much mass as
was there to begin with, the solution is no longer close to the linear flow, and
one expects blow-up.
\end{itemize}
The transition occurs when $\eps^{p-1}\Hp(T_\eps)\asymp1$, which is exactly the
law \eqref{eq:master} below. The weight $(1+s)^{-\alpha}$ expresses the fact
that a unit of reaction is worth less the later it acts, because by then the
heat flow has spread the solution out.

This heuristic is not a proof. It treats the solution as the linear flow at
all times, which is no longer true once the nonlinearity matters, and at the
critical power ($\alpha=1$ for $h\equiv1$) the contributions of all dyadic time
scales add up. Proving the upper bound for the lifespan therefore requires
summing the contributions of all scales, for an arbitrary $h$; this is the main
difficulty of the paper (see ``Method of proof'' below).

\subsection{Main results}

This subsection presents the main results of the paper. We begin with the
two-sided lifespan law and the critical amplitude (Theorems~\ref{thm:master}
and~\ref{thm:threshold}). We then turn to classical solutions: their
regularity (Theorem~\ref{thm:regularity}) and their a priori estimates
(Theorem~\ref{thm:apriori} and Corollary~\ref{cor:small}). Finally we list the
operators covered (Proposition~\ref{prop:examples}) and the explicit lifespan
laws that follow from the main theorem (Subsection~\ref{subsec:laws}). Each
statement is followed by a short summary in words.

\begin{theorem}\label{thm:master}
Suppose that Assumption~\ref{ass:H} and \eqref{eq:h-ass} hold. There are constants
$c_{\rm fp},C_{\rm up}>0$ and $T_1\ge1$, depending only on $X$, $p$ and $u_0$
\textup{(}not on $\eps$ or $h$\textup{)}, such that:
\begin{enumerate}[label=\textup{(\alph*)}]
\item if $\eps>0$, $T\in(0,\infty]$ and $\eps^{p-1}\Hp(T)\le c_{\rm fp}$,
then $T_\eps\ge T$ and
\begin{equation}\label{eq:X-bound}
 \sup_{0\le t<T}\bigl[\norm{u(t)}_1+(1+t)^{q/2}\norm{u(t)}_\infty\bigr]
 \le C\eps;
\end{equation}
\item for every $\eps>0$ and every $\tau\in[T_1,T_\eps)$,
\begin{equation}\label{eq:upper-large}
 \int_{T_1}^{\tau}h(t)\,t^{-\alpha}\,\dd t\le C_{\rm up}\,\eps^{-(p-1)},
 \quad\text{hence}\quad
 \Hp(\tau)\le\Hp(T_1)+C_{\rm up}\,\eps^{-(p-1)} .
\end{equation}
\end{enumerate}
Consequently, exactly one of the following holds.
\begin{enumerate}[label=\textup{(\roman*)}]
\item If $\Hp(\infty)<\infty$, then $T_\eps=\infty$ for all
$0<\eps\le(c_{\rm fp}/\Hp(\infty))^{1/(p-1)}$.
\item If $\Hp(\infty)=\infty$, then $T_\eps<\infty$ for every $\eps>0$, and
\begin{equation}\label{eq:master}
 c_{\rm fp}\,\eps^{-(p-1)}\le\Hp(T_\eps)\le2C_{\rm up}\,\eps^{-(p-1)}
 \qquad\text{for }0<\eps\le\eps_0 ,
\end{equation}
where $\eps_0:=(C_{\rm up}/\Hp(T_1))^{1/(p-1)}$ \textup{(}$\eps_0=\infty$ if $\Hp(T_1)=0$\textup{)}.
\end{enumerate}
\end{theorem}

\begin{inwords}
Part (a) is a guarantee: as long as $\eps^{p-1}\Hp(T)\le c_{\rm fp}$, the
solution exists at least up to time $T$ and stays comparable to the linear flow
(its $L^1$ norm stays of order $\eps$ and its height decays like
$(1+t)^{-q/2}$). Part (b) is the converse: whatever the amplitude, the weighted
reaction accumulated before the lifespan cannot exceed a constant multiple of
$\eps^{-(p-1)}$. Together they say that small data are global when the total
weighted reaction $\Hp(\infty)$ is finite, and that otherwise the lifespan is
the time at which $\Hp$ reaches the level $\eps^{-(p-1)}$, up to constants. No
regularity of $h$ is needed.
\end{inwords}

Since larger data cannot survive longer, there is also a critical amplitude.

\begin{theorem}\label{thm:threshold}
Suppose that Assumption~\ref{ass:H} and \eqref{eq:h-ass} hold, and let
\[
 \eps_*:=\sup\{\eps>0:\ T_\eps=\infty\}\in[0,\infty],\qquad \sup\emptyset:=0 .
\]
\begin{enumerate}[label=\textup{(\roman*)}]
\item \emph{Monotonicity.} $\eps\mapsto T_\eps$ is nonincreasing and
$u_{\eps_1}\le u_{\eps_2}$ for $\eps_1\le\eps_2$. Hence the solution is global
for $\eps<\eps_*$ and blows up in finite time for $\eps>\eps_*$.
\item \emph{Criterion for $\eps_*>0$.}
\begin{itemize}[nosep,leftmargin=1.6em]
\item If $\Hp(\infty)=\infty$, then $\eps_*=0$, i.e.\ every solution with
nontrivial data blows up in finite time.
\item If $\Hp(\infty)<\infty$, then $\eps_*\ge(c_{\rm fp}/\Hp(\infty))^{1/(p-1)}>0$,
i.e.\ $T_\eps=\infty$ whenever $\eps^{p-1}\Hp(\infty)\le c_{\rm fp}$. If moreover
$\Hp(\infty)\ge2\Hp(T_1)$, then also
$\eps_*\le(2C_{\rm up}/\Hp(\infty))^{1/(p-1)}$.
\end{itemize}
In particular, $\eps_*>0$ if and only if $\Hp(\infty)<\infty$.
\item \emph{Large data blow up.} $\eps_*<\infty$ unless $h=0$ a.e.\
\textup{(}in which case every solution is global\textup{)}.
\end{enumerate}
\end{theorem}

\begin{inwords}
Larger data cannot live longer, so there is a sharp transition amplitude
$\eps_*$: solutions are global below it and blow up above it. It is positive
exactly when $\Hp(\infty)<\infty$, so the same quantity $\Hp$ decides whether
small data are global at all. When $\Hp(\infty)<\infty$ and $\Hp(T_1)$ is not
dominant, $\eps_*$ is comparable to $\Hp(\infty)^{-1/(p-1)}$.
\end{inwords}

\begin{remark}\label{rem:comments}
Theorems~\ref{thm:master} and~\ref{thm:threshold} motivate the
following comments.
\begin{enumerate}[label=\textup{(\alph*)}]
\item No regularity, monotonicity or doubling is imposed on $h$: it may
oscillate, vanish on long intervals, or concentrate in short windows. This is
the reason for stating the law through $\Hp(T_\eps)$ rather than
$\Hp^{-1}(\eps^{-(p-1)})$, since $\Hp$ may have plateaus. The constants
$c_{\rm fp}$ and $C_{\rm up}$ come from unrelated arguments, so we obtain the
order of the lifespan but not the sharp multiplicative constant.
\item The second case of Theorem~\ref{thm:threshold}-(ii),
``$\Hp(\infty)<\infty$ $\Rightarrow$ some data are global'', is the
contrapositive of the necessary condition for blow-up in
\cite[Thm.~4.2]{ChatzakouKassymovRuzhansky2026}. Our setting is also covered by
the abstract framework of Meglioli and Punzo \cite{MeglioliPunzo2026}: $\R^n$
with the distance $\CC$ and Lebesgue measure carries the regular, strongly
local Dirichlet form $\int\sum_j|X_ju|^2\dd x$, whose heat kernel is
conservative and satisfies the two-sided Gaussian bounds
\eqref{eq:Gaussian-general}. For continuous $h>0$ and continuous data,
\cite[Thm.~6.6]{MeglioliPunzo2026} gives global existence for small data when
$\int_0^\infty h(t)\sup_{x,y}\Gamma_{t+\gamma}(x,y)^{p-1}\dd t<\infty$, which
here is exactly $\Hp(\infty)<\infty$ (for data below a multiple of a
heat-kernel column $\Gamma_\gamma(\cdot,y_0)$); their single-time criterion
\cite[Thm.~3.3]{MeglioliPunzo2026} (see also \cite{LenzSchmidtZimmermann2023}
for $h\equiv1$) contains the large-data blow-up of
Theorem~\ref{thm:threshold}-(iii) and Proposition~\ref{prop:window}; and
\cite[Thm.~6.4]{MeglioliPunzo2026} gives blow-up at $p=p_F$ for $h\equiv1$.
For a general $h$ the blow-up criterion \cite[Thm.~6.2]{MeglioliPunzo2026} is
single-scale; in our setting it reads
$\limsup_{T\to\infty}T^{-\alpha}\int_0^Th\dd t=\infty$. The new points are the
converse $\Hp(\infty)=\infty\Rightarrow T_\eps<\infty$ for all $\eps>0$, for an
arbitrary $h\ge0$ in $L^1_{\rm loc}$, which requires summing the reaction over
all scales; the two-sided law \eqref{eq:master}; and the quantitative
consequences in Theorems~\ref{thm:powerlog}--\ref{thm:nonstandard}.
\item The geometry enters only through the homogeneous dimension $q$; examples
are collected in Subsection~\ref{subsec:examples}.
\end{enumerate}
\end{remark}

\subsection*{Classical solutions}

The solutions in Theorems~\ref{thm:master} and~\ref{thm:threshold} are mild
solutions (Definition~\ref{def:mild}), for which nothing is assumed on $h$
beyond \eqref{eq:h-ass}. When $h$ is continuous, it is natural to ask whether
they solve the equation pointwise. Let $h\in C((0,\infty))$ satisfy
\eqref{eq:h-ass}. A \emph{classical solution on $[0,T)$} is a function $u\ge0$
with $u\in C([0,T);L^1(\R^n))$ and $u(0)=\eps u_0$, bounded on $[0,t]\times\R^n$
for every $t<T$, such that $u_t$, $X_ku$ and $X_iX_ju$
\textup{(}$1\le i,j,k\le m$\textup{)} exist and are continuous on
$(0,T)\times\R^n$ and bounded on $[t_0,t_1]\times\R^n$ for $0<t_0<t_1<T$, and
$u_t-\Lcal u=h\,u^p$ pointwise in $(0,T)\times\R^n$.

\begin{theorem}[Classical solutions]\label{thm:regularity}
Suppose that Assumption~\ref{ass:H} and \eqref{eq:h-ass} hold, let $u$ be the mild
solution on $[0,T_\eps)$, and let $I\subset(0,T_\eps)$ be an open interval. Then the following hold.
\begin{enumerate}[label=\textup{(\alph*)}]
\item \emph{Positivity.} $u\ge\eps P_tu_0>0$ and $u\in C((0,T_\eps)\times\R^n)$.
\item \emph{Smooth reaction.} If $h\in C^\infty(I)$, then
$u\in C^\infty(I\times\R^n)$ and $u_t-\Lcal u=hu^p$ pointwise in $I\times\R^n$.
\item \emph{H\"older reaction.} If $h\in C^{0,\gamma}_{\rm loc}(I)$,
$0<\gamma\le1$, and $\varkappa\in(0,1)$, $\varkappa\le2\gamma$, then for every
compact $J\subset I$ the derivatives $X_ku$, $u_t$ and $X_iX_ju$
\textup{(}all $i,j,k$\textup{)} exist, are continuous and bounded on
$J\times\R^n$, $u_t=\Lcal u+hu^p$ pointwise, and each $D\in\{u_t,X_iX_ju\}$
satisfies
\begin{equation}\label{eq:schauder}
 |D(t,x)-D(s,y)|\le C\bigl(\CC(x,y)^\varkappa+|t-s|^{\varkappa/2}\bigr),
 \qquad t,s\in J,\ x,y\in\R^n .
\end{equation}
\item \emph{Classical solutions.} If $h\in C^{0,\gamma}_{\rm loc}((0,\infty))$
for some $\gamma\in(0,1]$ \textup{(}in particular if $h\in C^\infty((0,\infty))$\textup{)},
then $u$ is a classical solution on $[0,T_\eps)$. Conversely, every classical
solution on $[0,T)$ coincides with $u$ on $[0,T)$, so $T\le T_\eps$ and $T_\eps$
is the maximal existence time of classical solutions; if $T_\eps<\infty$, then
$\limsup_{t\uparrow T_\eps}\norm{u(t)}_\infty=\infty$.
\end{enumerate}
\end{theorem}

\begin{inwords}
The solution is positive, and it is as smooth as the coefficient allows: it is
$C^\infty$ where $h$ is, and for H\"older continuous $h$ one gets a
Schauder-type estimate, namely that $u_t$ and the second-order derivatives
$X_iX_ju$ are H\"older continuous in the parabolic scale of $\CC$. For H\"older
$h$ the mild solution is a classical solution, and it is the only one.
\end{inwords}

By (d), for H\"older $h$ Theorems~\ref{thm:master} and~\ref{thm:threshold} are
statements about classical solutions: $T_\eps$ is the lifespan of the classical
solution. Part (c) is the Schauder regularity for $\partial_t-\Lcal$ in the
parabolic scale of $\CC$; in particular the commutators $[X_i,X_j]u$ are
H\"older continuous too. Local Schauder estimates for parabolic operators
built on H\"ormander vector fields are due to Bramanti and Brandolini
\cite{BramantiBrandolini2007}, and for homogeneous fields global heat kernels
with $C^{2,\alpha}_X$ regularity and Gaussian bounds were constructed by Biagi
and Bramanti \cite{BiagiBramanti2021}. We give a self-contained proof adapted
to the Duhamel formula, in which $h$ is only locally integrable in the past and
H\"older continuous near $J$. When $h$ jumps, as in the intermittent examples
below, $u$ is not $C^1$ in time and only the mild formulation is available;
this is why Theorems~\ref{thm:master} and~\ref{thm:threshold} are stated for
mild solutions. The constant $C$ in \eqref{eq:schauder} depends on $J,I,\gamma,
\varkappa,X,p,\eps,u_0$, $\norm h_{L^1(0,\max J)}$, the $C^{0,\gamma}$ norm of
$h$ near $J$ and $\sup_{[0,\max J]}\norm u_\infty$.

\subsection*{A priori estimates}

By an \emph{a priori estimate} we mean a bound that holds for \emph{every}
classical solution on $[0,T)$ and uses only the equation: no smallness of $\eps$
and nothing from the construction in Theorem~\ref{thm:master}. Such solutions
exist for H\"older $h$ (Theorem~\ref{thm:regularity}-(d)), and every classical
solution is a mild solution (Lemma~\ref{lem:classical-mild}), so
Theorems~\ref{thm:master} and~\ref{thm:threshold} apply to it. We measure
solutions by the two quantities

\[
 N_T(u):=\sup_{0\le t<T}\bigl[\norm{u(t)}_1+(1+t)^{q/2}\norm{u(t)}_\infty\bigr],
 \quad
 \Theta_\eps(t):=(p-1)\bigl(\eps\norm{u_0}_\infty\bigr)^{p-1}\!\int_0^th(s)\dd s .
\]
Note that $N_T(u)$ is finite for $T<T_\eps$.

\begin{theorem}[A priori estimates for classical solutions]\label{thm:apriori}
Suppose that Assumption~\ref{ass:H} and \eqref{eq:h-ass} hold and that $h\in C((0,\infty))$, and let
$u$ be a classical solution on $[0,T)$. No smallness of $\eps$ is assumed. Then
$T\le T_\eps$ and, for $0\le t<T$:
\begin{enumerate}[label=\textup{(\alph*)}]
\item \emph{Maximum principle and mass.} $u(t)\ge\eps P_tu_0$ and
\[\eps\norm{u_0}_1\le\norm{u(t)}_1=\eps\norm{u_0}_1+\int_0^th\norm{u(s)}_p^p\dd s.\]
\item \emph{$L^k$ bounds.} For $1\le k\le\infty$,
$\norm{u(t)}_k\le\eps\norm{u_0}_k
\exp\bigl(\int_0^th\norm{u}_\infty^{p-1}\dd s\bigr)$.
\item \emph{ODE bound.} If $\Theta_\eps(t)<1$, then
$0\le u(t,x)\le\eps\norm{u_0}_\infty(1-\Theta_\eps(t))^{-\frac1{p-1}}$ and, for
$1\le k\le\infty$,
\begin{equation}\label{eq:Lk-ODE}
 \norm{u(t)}_k\le\eps\norm{u_0}_k\bigl(1-\Theta_\eps(t)\bigr)^{-\frac1{p-1}} .
\end{equation}
\end{enumerate}
Moreover $T_\eps\ge T_*:=\sup\{t:\Theta_\eps(t)<1\}$.
\end{theorem}

\begin{inwords}
(a) The solution lies above the linear flow, and its total mass grows by exactly
the integral of $\norm{u}_p^p$ weighted by $h$. (b) Every $L^k$ norm can grow
at most by the factor $\exp\bigl(\int_0^th\norm u_\infty^{p-1}\bigr)$. (c)
Comparing with the ordinary differential equation $v'=hv^p$ shows that as long
as the ``ODE time'' $\Theta_\eps(t)$ is below $1$, the solution stays bounded by
$\eps\norm{u_0}_\infty(1-\Theta_\eps)^{-1/(p-1)}$; in particular the solution
cannot blow up before $T_*$.
\end{inwords}

Theorem~\ref{thm:apriori} uses only the maximum principle and the
$L^k$-contractivity of the semigroup. It is far from sharp for large $t$, since
it sees no decay, but it holds for arbitrary $\eps$ and $h$ and already gives
the elementary lifespan bound $T_\eps\ge T_*$. The decay and the closeness to
the linear flow come from Theorem~\ref{thm:master}-(a) and are recorded in the
following corollary.

\begin{corollary}[Small-data regime]\label{cor:small}
Suppose that Assumption~\ref{ass:H} and \eqref{eq:h-ass} hold and that $h\in C((0,\infty))$, and
let $u$ be a classical solution on $[0,T)$ with $\eps^{p-1}\Hp(T)\le c_{\rm fp}$.
Then $N_T(u)\le C\eps$, and for $0\le t<T$ and $1\le k\le\infty$, with
$C=C(X,p,u_0)$,
\[
 \norm{u(t)}_k\le C\eps(1+t)^{-\frac q2(1-\frac1k)},\qquad
 \norm{u(t)-\eps P_tu_0}_k\le C\eps^p\Hp(T)(1+t)^{-\frac q2(1-\frac1k)},
\]
and $\eps\norm{u_0}_1\le\norm{u(t)}_1\le\eps\norm{u_0}_1+C\eps^p\Hp(t)$.
\end{corollary}

\begin{inwords}
In the small-data regime the solution decays like the heat flow, and it differs
from the linear flow $\eps P_tu_0$ by a term that is smaller than $\eps$ by
the factor $\eps^{p-1}\Hp(T)\le c_{\rm fp}$.
\end{inwords}

\subsection{Particular operators}\label{subsec:examples}

Assumption~\ref{ass:H} contains many classical operators; for each of them
Theorems~\ref{thm:master}--\ref{thm:nonstandard} hold with the value of $q$
below. Recall that $p_F=1+2/q$ is then the Fujita exponent of $h\equiv1$.

\begin{proposition}\label{prop:examples}
Each operator in Table~\ref{tab:operators} satisfies Assumption~\ref{ass:H}
with the indicated weights $\sigma$ and homogeneous dimension $q=\sum\sigma_i$.
\end{proposition}

\begin{table}[htbp]
\centering
\small
\caption{Sums of squares covered by the theorems.}
\label{tab:operators}
\renewcommand{\arraystretch}{1.45}
\begin{tabular}{@{}llllc@{}}
\toprule
Operator & space & weights $\sigma$ & $q$ & $p_F$ \\
\midrule
Laplacian $\Delta$ & $\R^n$ & $(1^n)$ & $n$ & $1+\dfrac2n$ \\
Heisenberg $\Delta_H$ & $\R^{2N+1}$ & $(1^{2N},2)$ & $2N+2$ & $1+\dfrac1{N+1}$ \\
Engel group & $\R^4$ & $(1,1,2,3)$ & $7$ & $\dfrac97$ \\
Cartan group (free, step 3) & $\R^5$ & $(1,1,2,3,3)$ & $10$ & $\dfrac65$ \\
Free step-2, $m$ generators & $\R^{m+\binom m2}$ & $(1^m,2^{\binom m2})$ & $m^2$ & $1+\dfrac2{m^2}$ \\
H-type, $\dim V_1=m,\dim V_2=k$ & $\R^{m+k}$ & $(1^m,2^k)$ & $m+2k$ & $1+\dfrac2{m+2k}$ \\
Carnot, layers $V_1,\dots,V_s$ & $\R^n$ & $(j^{\dim V_j})_j$ & $\sum_jj\dim V_j$ & $1+\dfrac2q$ \\
Grushin $\Delta_x+|x|^{2g}\Delta_y$ & $\R^{N+k}$ & $(1^N,(g{+}1)^k)$ & $N+(g{+}1)k$ & $1+\dfrac2q$ \\
Martinet sum of squares & $\R^3$ & $(1,1,3)$ & $5$ & $\dfrac75$ \\
Product of two such operators & $\R^{n_1+n_2}$ & concatenated & $q_1+q_2$ & $1+\dfrac2q$ \\
\bottomrule
\end{tabular}
\end{table}

\begin{proof}
In each case one checks that the fields are polynomial, homogeneous of degree
one, and that iterated commutators span the tangent space at $0$.

\emph{Euclidean.} $X_j=\partial_{x_j}$, all weights equal one.

\emph{Heisenberg.} $X_j=\partial_{x_j}+2y_j\partial_z$,
$Y_j=\partial_{y_j}-2x_j\partial_z$ are homogeneous for
$(x,y,z)\mapsto(\lambda x,\lambda y,\lambda^2z)$, and $[X_j,Y_j]=-4\partial_z$.

\emph{Engel.} $X_1=\partial_1$, $X_2=\partial_2+x_1\partial_3+\tfrac12x_1^2\partial_4$
with weights $(1,1,2,3)$. Then $[X_1,X_2]=\partial_3+x_1\partial_4$ and
$[X_1,[X_1,X_2]]=\partial_4$, so the rank is $4$ at every point.

\emph{Carnot groups.} Let $G$ be a stratified group with
$\mathfrak g=V_1\oplus\dots\oplus V_s$, $[V_1,V_j]=V_{j+1}$, and let
$X_1,\dots,X_m$ be a basis of $V_1$. In exponential coordinates the $X_j$ are
polynomial and homogeneous of degree one for the dilations with weight $j$ on
$V_j$, and they satisfy H\"ormander's condition because $V_1$ generates
$\mathfrak g$ \cite{Folland1975,VaropoulosSaloffCosteCoulhon1992}; $q=\sum_jj\dim V_j$.
This covers the Heisenberg and Engel groups, the Cartan group, the free
step-2 group ($\dim V_2=\binom m2$, hence $q=m+m(m-1)=m^2$), and H-type
groups.

\emph{Grushin-type.} The operator $\Delta_x+|x|^{2g}\Delta_y$, $g\in\mathbb{N}$, on $\R^N\times\R^k$
is homogeneous of degree two under $(x,y)\mapsto(\lambda x,\lambda^{g+1}y)$.
By the multinomial identity $|x|^{2g}=\sum_{|\nu|=g}\frac{g!}{\nu!}x^{2\nu}$ it equals
$\sum_i\partial_{x_i}^2+\sum_{|\nu|=g,\,j}\bigl((g!/\nu!)^{1/2}x^\nu\partial_{y_j}\bigr)^2$.
These fields are homogeneous of degree one, and $g$ commutators with
the $\partial_{x_i}$ give $\partial_{y_j}$, which is the rank condition.

\emph{Martinet.} $X_1=\partial_x$, $X_2=\partial_y+x^2\partial_z$ with weights
$(1,1,3)$; $[X_1,X_2]=2x\partial_z$ vanishes at $0$, but
$[X_1,[X_1,X_2]]=2\partial_z$, so the rank at $0$ is $3$.

\emph{Products.} If $\Lcal_1,\Lcal_2$ satisfy the assumption on $\R^{n_1},\R^{n_2}$,
then $\Lcal_1+\Lcal_2$ on $\R^{n_1+n_2}$ does, with the concatenated weights
(after relabelling the coordinates so that the weights increase, as required in
\eqref{eq:dil}) and $q=q_1+q_2$.
\end{proof}

We add three comments on Proposition~\ref{prop:examples}.

\smallskip
\noindent\emph{Equal dimension, equal law.} The lifespan laws depend on $X$
only through $q$. For example, $\R^4$ and the Heisenberg group $\mathbb H^1$
both have $q=4$, hence identical lifespan laws for every $h$, although their
diffusions are very different. On the Engel group ($q=7$) with $h\equiv1$ and
$1<p<\frac97$ one gets $T_\eps\asymp\eps^{-2(p-1)/(2-7(p-1))}$, and on
$\mathbb H^N$ with $h\asymp(1+t)^\beta$ the critical exponent is
$1+(1+\beta)/(N+1)$.

\smallskip
\noindent\emph{Beyond group laws.} Euclidean space and Carnot groups carry a
group law for which the fields are left invariant. The Grushin and Martinet
operators do not come from a Carnot group on the same space; they are
projections of Carnot operators through the global lifting of
\cite{BiagiBonfiglioli2017,BiagiBonfiglioli2024}, and the heat-kernel bounds
\eqref{eq:Gaussian-general} still hold for them. Since our proofs use only
\eqref{eq:Gaussian-general}, no group structure on $\R^n$ is required. Grushin
operators with non-integer exponent, e.g.\ $\partial_x^2+|x|^{2\gamma}\partial_y^2$
with $\gamma\notin\mathbb N$, have non-smooth fields and are not covered.

\smallskip
\noindent\emph{Kolmogorov-type operators.} The Kolmogorov operator
$\partial_t-\Delta_x-\langle x,\nabla_y\rangle$ on $\R^N_x\times\R^N_y$ contains
a drift $X_0=\langle x,\nabla_y\rangle$, so the equation
$u_t-\Delta_xu-\langle x,\nabla_y\rangle u=h(t)u^p$ is not of the form
\eqref{eq:main}: the semigroup is not self-adjoint, and Assumption~\ref{ass:H}
describes sums of squares only. The operator is homogeneous for
$(x,y)\mapsto(\lambda x,\lambda^3y)$ with homogeneous dimension $q=N+3N=4N$,
its Gaussian kernel satisfies $\norm{P_t}_{L^1\to L^\infty}\asymp t^{-2N}=t^{-q/2}$,
and since the drift is divergence free, $P_t$ is Markov and Lebesgue measure is
invariant. The lower estimate of Theorem~\ref{thm:master}-(a), and hence
small-data global existence for $\Hp(\infty)<\infty$, uses only positivity,
contractivity on $L^1$ and $L^\infty$, strong continuity on $L^1$ and the decay
$t^{-q/2}$; therefore it applies with $q=4N$. We expect the upper estimate to
extend as well, with $\CC$ replaced by the Kolmogorov quasi-norm, but we have
not checked the details and we do not claim the two-sided law in this case. The
same remark applies to the other homogeneous Kolmogorov--Fokker--Planck
operators of Lanconelli and Polidoro type \cite{LanconelliPolidoro1994}.

\subsection{Explicit lifespan laws}\label{subsec:laws}

Whenever $h$ has a recognizable shape, the law \eqref{eq:master} can be turned
into an explicit formula for $T_\eps$ in two elementary steps: compute
$\Hp(T)\asymp G(T)$ for large $T$, and then solve $G(T_\eps)\asymp\eps^{-(p-1)}$
(Section~\ref{sec:applications}). Table~\ref{tab:examples} collects
representative cases.

\smallskip
\noindent\emph{A worked example.} Let $h\equiv1$ and $\alpha<1$, i.e.\ $p<1+2/q$.
Then $\Hp(T)=\int_0^T(1+t)^{-\alpha}\dd t\asymp T^{1-\alpha}$, and
$T_\eps^{1-\alpha}\asymp\eps^{-(p-1)}$ gives
$T_\eps\asymp\eps^{-(p-1)/(1-\alpha)}$, which is the polynomial law
$\eps^{-2(p-1)/(2-q(p-1))}$ in Theorem~\ref{thm:powerlog}-(a) below. If
$\alpha=1$, then $\Hp(T)=\log(1+T)$ and one finds the exponential law
$\log T_\eps\asymp\eps^{-(p-1)}$; if $\alpha>1$, then $\Hp(\infty)<\infty$ and
small data are global.

\smallskip
\noindent\emph{Conventions.} Below, $p_F:=1+2/q$; an asymptotic relation on $h$
is assumed for $t\ge T_h$ with some $T_h\ge e$; statements about $T_\eps$ are
for small $\eps$; and $L_1(t)=\log(e+t)$, $L_{j+1}=\log(e+L_j)$ are the iterated
logarithms.

\begin{theorem}\label{thm:powerlog}
Suppose that Assumption~\ref{ass:H} and \eqref{eq:h-ass} hold.
\begin{enumerate}[label=\textup{(\alph*)}]
\item \emph{$h\equiv1$.} The Fujita exponent is $p_F=1+2/q$, and
\[
 T_\eps\asymp\eps^{-\frac{2(p-1)}{2-q(p-1)}}\ \ (1<p<p_F),\qquad
 \log T_\eps\asymp\eps^{-(p-1)}\ \ (p=p_F),
\]
while small data are global for $p>p_F$.
\item \emph{$h\asymp(1+t)^\beta$.} With $\rho:=1+\beta-\alpha$:
$T_\eps\asymp\eps^{-(p-1)/\rho}$ if $\rho>0$; $\log T_\eps\asymp\eps^{-(p-1)}$
if $\rho=0$; small data are global if $\rho<0$. Thus the critical exponent is
$p_F(\beta)=1+2(1+\beta)/q$ for $\beta>-1$, and for $\beta\le-1$ small data
are global for all $p>1$.
\item \emph{$h\asymp(1+t)^\beta[\log(e+t)]^\gamma$, $\rho>0$.}
$T_\eps\asymp\eps^{-\frac{p-1}\rho}(\log\frac1\eps)^{-\frac\gamma\rho}$;
for $\rho<0$ small data are global. The case $\rho=0$ is part \textup{(d)}
with $j_*=1$.
\item \emph{Critical hierarchy.} Let $\beta>-1$, $p=p_F(\beta)$, $j_*\ge1$,
$\chi\in\R$ and
$h\asymp(1+t)^\beta\prod_{j<j_*}L_j(t)^{-1}\,L_{j_*}(t)^\chi$. If $\chi>-1$,
then $L_{j_*}(T_\eps)\asymp\eps^{-(p-1)/(\chi+1)}$; if $\chi=-1$, then
$L_{j_*+1}(T_\eps)\asymp\eps^{-(p-1)}$; if $\chi<-1$, small data are global.
In particular $h\asymp(1+t)^\beta/L_1$ gives
$\log\log T_\eps\asymp\eps^{-(p-1)}$, and $h\asymp(1+t)^\beta/(L_1L_2)$ gives
$\log\log\log T_\eps\asymp\eps^{-(p-1)}$.
\end{enumerate}
\end{theorem}

\begin{inwords}
Parts (a) and (b) give polynomial lifespans below the critical exponent and
lifespans of exponential type, $\log T_\eps\asymp\eps^{-(p-1)}$, at the
critical exponent $p_F(\beta)=1+2(1+\beta)/q$; above it small data are global.
Part (c) shows that a logarithmic factor in $h$ changes the polynomial law only
by a logarithmic factor. Part (d) shows what happens at the critical exponent
when $h$ carries iterated logarithms: at the borderline case $\chi=-1$, each
additional logarithm in the denominator of $h$ adds one more exponential to the
lifespan.
\end{inwords}

\begin{remark}[The autonomous case $h\equiv1$]\label{rem:autonomous}
For $h\equiv1$, Theorem~\ref{thm:powerlog}-(a) states that $p_F=1+2/q$ and
\[
 T_\eps\asymp\eps^{-\left(\frac1{p-1}-\frac q2\right)^{-1}}\ \ (1<p<p_F),
 \qquad
 \log T_\eps\asymp\eps^{-(p-1)}\ \ (p=p_F).
\]
For $\Lcal=\Delta$ on $\R^n$ ($q=n$) these are the lifespan orders of Lee and
Ni \cite{LeeNi1992}, with $p_F=1+2/n$ going back to Fujita \cite{Fujita1966}.
On the Heisenberg group $\mathbb H^N$ ($q=2N+2$, so $q/2=N+1$ and
$p_F=1+1/(N+1)$, the Fujita exponent of Zhang \cite{Zhang1998}) they are the
polynomial and exponential lifespans of Georgiev and Palmieri
\cite{GeorgievPalmieri2021}. For the Grushin operator
$\Delta_x+|x|^2\Delta_y$ on $\R^N\times\R^k$ ($q=N+2k$) we obtain the
analogous orders with $p_F=1+2/(N+2k)$ (see \cite{OliveiraViana2024,Kogoj1}
for the well-posedness theory and \cite{ChatzakouKassymovRuzhansky2026} for
the critical exponent); see also \cite{FinoKirane2026,FinoViana2026a} for
related Grushin problems.
\end{remark}

\begin{theorem}\label{thm:fast}
Suppose that Assumption~\ref{ass:H} and \eqref{eq:h-ass} hold.
\begin{enumerate}[label=\textup{(\alph*)}]
\item \emph{Exponential growth.} If $h\asymp e^{\theta t}(1+t)^\beta$,
$\theta>0$, then $T_\eps<\infty$ for all $\eps>0$ and
$T_\eps=\frac{p-1}\theta\log\frac1\eps+\frac{\alpha-\beta}\theta
\log\log\frac1\eps+O(1)$.
\item \emph{Integrable reactions.} If $h(t)(1+t)^{-\alpha}\in L^1(0,\infty)$,
small data are global. This holds for all $p>1$ if $h\le Ce^{-\theta t}(1+t)^\beta$
with $\theta>0$ or if $h$ has bounded support $\subset[0,T_{\rm off}]$; then
$u(t)=P_{t-T_{\rm off}}u(T_{\rm off})$ for $t\ge T_{\rm off}$.
\item \emph{Periodic reactions.} If $h\ge0$ is $\varpi$-periodic with positive
mean, then $\Hp(T)\asymp\int_1^T(1+t)^{-\alpha}\dd t$ for large $T$ when the
latter diverges, and $\Hp(\infty)<\infty$ otherwise: the laws are those of
\textup{(a)} in Theorem~\ref{thm:powerlog}.
\item \emph{Dense blocks.} For $h=\one_{\mathcal A}$,
$\mathcal A=\bigcup_{k\ge1}[2^{2k},2^{2k+1}]$, the Fujita exponent and the
lifespan orders are again those of Theorem~\ref{thm:powerlog}-(a).
\item \emph{Sparse blocks.} For $h=\one_{\mathcal A_\vartheta}$,
$\mathcal A_\vartheta=\bigcup_{k\ge1}[2^k,2^k+2^{\vartheta k}]$,
$0\le\vartheta<1$: $T_\eps\asymp\eps^{-(p-1)/(\vartheta-\alpha)}$ if
$\alpha<\vartheta$, $\log T_\eps\asymp\eps^{-(p-1)}$ if $\alpha=\vartheta$,
and small data are global if $\alpha>\vartheta$. The effective threshold is
$1+2\vartheta/q$.
\end{enumerate}
Moreover, for every $h$, if $h=0$ a.e.\ on $(t_1,t_2)$, then
$T_\eps\notin(t_1,t_2]$.
\end{theorem}

\begin{inwords}
If $h$ grows exponentially, the lifespan is only logarithmic in $1/\eps$:
$T_\eps\approx\frac{p-1}{\theta}\log\frac1\eps$. Reactions that decay fast
enough, or are switched off after a finite time, are harmless for small data.
Periodic reactions with positive mean and dense blocks behave like $h\equiv1$.
For sparse blocks the active intervals are so thin that the effective critical
exponent is $1+2\vartheta/q$ instead of $1+2/q$. Finally, blow-up can never
occur while the reaction is switched off.
\end{inwords}

\begin{table}[htbp]
\centering
\small
\caption{Representative laws; $\rho=1+\beta-\alpha$.}
\label{tab:examples}
\begin{tabular}{p{0.28\textwidth}p{0.24\textwidth}p{0.36\textwidth}}
\toprule
$h(t)$ & Regime & Small-data lifespan \\
\midrule
$1$ & $\alpha<1$ & $T_\eps\asymp\eps^{-(p-1)/(1-\alpha)}$ \\
$1$ & $\alpha=1$ & $\log T_\eps\asymp\eps^{-(p-1)}$ \\
$(1+t)^\beta$ & $\rho>0$ & $T_\eps\asymp\eps^{-(p-1)/\rho}$ \\
$(1+t)^\beta$ & $\rho=0$ & $\log T_\eps\asymp\eps^{-(p-1)}$ \\
$(1+t)^\beta\log^\gamma(e+t)$ & $\rho>0$ & $T_\eps\asymp\eps^{-(p-1)/\rho}\log^{-\gamma/\rho}\frac1\eps$ \\
$(1+t)^\beta\log^\gamma(e+t)$ & $\rho=0,\ \gamma>-1$ & $\log T_\eps\asymp\eps^{-(p-1)/(\gamma+1)}$ \\
$(1+t)^\beta/\log(e+t)$ & $\rho=0$ & $\log\log T_\eps\asymp\eps^{-(p-1)}$ \\
$e^{\theta t}(1+t)^\beta$ & always & $T_\eps\asymp\log\frac1\eps$ \\
periodic, positive mean & & as for $h\equiv1$ \\
$\one_{\mathcal A_\vartheta}$ & $\alpha<\vartheta$ & $T_\eps\asymp\eps^{-(p-1)/(\vartheta-\alpha)}$ \\
$\one_{\mathcal A_\vartheta}$ & $\alpha=\vartheta$ & $\log T_\eps\asymp\eps^{-(p-1)}$ \\
\bottomrule
\end{tabular}
\end{table}

The weight $(1+t)^{-\alpha}$ in $\Hp$ says that a unit of reaction is worth
less the later it acts, because the heat flow has spread the solution out.
Since $h$ is only locally integrable, this produces effects that no coefficient
$(1+t)^\beta$ can show. The next theorem collects three of them. The
heuristic of Subsection~\ref{subsec:heuristic} predicts the first one: a pulse
of mass $M$ at time $T_0$ meets a solution of height about
$\eps(1+T_0)^{-q/2}$, so it matters only when
$M\bigl(\eps(1+T_0)^{-q/2}\bigr)^{p-1}$ is of order one, that is, when
$\eps\approx(1+T_0)^{q/2}M^{-1/(p-1)}$.

\begin{theorem}\label{thm:nonstandard}
Suppose that Assumption~\ref{ass:H} and \eqref{eq:h-ass} hold; constants depend only on
$X$, $p$, $u_0$. Let $\eps_*$ be as in Theorem~\ref{thm:threshold}.
\begin{enumerate}[label=\textup{(\roman*)}]
\item \emph{Delayed pulse.} Let $\supp h\subset[T_0,T_0+1]$ and
$M:=\int h\in(0,\infty)$. Then $\Hp(\infty)<\infty$ and
$c\,(1+T_0)^{q/2}M^{-\frac1{p-1}}\le\eps_*\le C\,(1+T_0)^{q/2}M^{-\frac1{p-1}}$.
So $\eps_*$ depends on the mass $M$ and not on its distribution in the window,
and a delay $T_0$ multiplies $\eps_*$ by $(1+T_0)^{q/2}$ for every $p>1$.
\item \emph{Sparse bursts.} For $h=\one_{\mathcal A_0}$ (unit bursts at
$[2^k,2^k+1]$) small data are global for \emph{every} $p>1$, although
$h\not\to0$ and $\int h=\infty$. For $h=\one_{\mathcal A_\vartheta}$ with
$0<\vartheta<1$ and $\alpha\le\vartheta$, for every $\eps>0$
\[
 T_\eps\in\bigcup_{k\ge1}\bigl(2^k,2^k+2^{\vartheta k}\bigr]:
\]
blow-up can only occur while the reaction is switched on, and this set has
relative size $\to0$ in $[2^k,2^{k+1}]$.
\item \emph{Staircase lifespan.} Let $1<p<p_F$, $s_k:=2^{2^k}$ and
$h=\one_{\mathcal A}$, $\mathcal A=\bigcup_{j\ge1}[s_{2j},s_{2j+1})$. Then
$\Hp(\infty)=\infty$ and
\[
 \liminf_{\eps\to0}\frac{\log T_\eps}{\log(1/\eps)}=\frac{p-1}{1-\alpha},
 \qquad
 \limsup_{\eps\to0}\frac{\log T_\eps}{\log(1/\eps)}\ge\frac{2(p-1)}{1-\alpha}.
\]
So there is no $\kappa$ with $T_\eps\asymp\eps^{-\kappa}$. Indeed there are
intervals $[\eps_j^-,\eps_j^+]$ with $\eps_j^+/\eps_j^-\ge s_{2j}^{\,c_*}$,
$c_*>0$, on which $s_{2j}<T_\eps\le2s_{2j}$.
\end{enumerate}
\end{theorem}

Items (i)--(ii) show that the naive heuristic ``the reaction wins when $\int h$
is large'' fails for concentrated $h$: what counts is the weighted integral
$\int h(t)(1+t)^{-\alpha}\dd t$. Item (iii) shows that a two-sided law for
$\Hp(T_\eps)$ cannot in general be turned into a power law for $T_\eps$: when
the active blocks are very far apart, the lifespan stays almost constant over a
whole range of amplitudes and then jumps.

\begin{remark}[What is new already for $\Lcal=\Delta$ on $\R^n$]\label{rem:new-Rn}
Several of the statements above are classical when $q=n$: the Fujita exponent
and the lifespan orders for $h\equiv1$ (Remark~\ref{rem:autonomous}), the
threshold $1+2(1+\beta)/n$ for $h\asymp(1+t)^\beta$ \cite{Meier1988,Meier1990},
and, for continuous positive $h$, the qualitative criteria of
\cite{MeglioliPunzo2026}: small-data global existence when
$\Hp(\infty)<\infty$, single-time blow-up criteria, and blow-up at $p=p_F$ for
$h\equiv1$ (see Remark~\ref{rem:comments}-(b)). As far as we could check, the
following remain new even in $\R^n$:
\begin{enumerate}[label=\textup{(\arabic*)}]
\item the two-sided law \eqref{eq:master} for \emph{every} $h\ge0$ in
$L^1_{\rm loc}$, with no continuity, monotonicity or regular variation, and
the dichotomy $\Hp(\infty)<\infty$ versus $\Hp(\infty)=\infty$ that it
contains, in particular blow-up for every $\eps>0$ when $\Hp(\infty)=\infty$
(for continuous positive $h$ the half $\Hp(\infty)<\infty\Rightarrow\eps_*>0$
is in \cite{MeglioliPunzo2026});
\item the critical amplitude $\eps_*$ and its two-sided control for delayed
pulses (Theorem~\ref{thm:nonstandard}-(i));
\item the critical hierarchy, i.e.\ the double- and triple-exponential
lifespans created by logarithmic factors in $h$ (Theorem~\ref{thm:powerlog}-(d));
\item the exponential-growth, periodic and intermittent laws of
Theorem~\ref{thm:fast}, and the effects of Theorem~\ref{thm:nonstandard}-(ii),(iii):
no Fujita exponent for sparse bursts, and a lifespan with no power law in
$\eps$;
\item the two-sided lifespan law of Section~\ref{sec:general} for
nonlinearities comparable to a power (qualitative criteria for such $f$ are in
\cite[Sect.~7]{MeglioliPunzo2026}).
\end{enumerate}
Our literature search was not exhaustive, and the list reflects what we found,
not a proof of priority.
\end{remark}

\subsection*{Method of proof}

\emph{Lower bound for the lifespan} (Theorem~\ref{thm:master}-(a)). We solve the
integral equation \eqref{eq:mild} by the contraction mapping principle in the
space of functions with
$\sup_t\{\norm{u(t)}_1+(1+t)^{q/2}\norm{u(t)}_\infty\}<\infty$. This norm is
adapted to the decay of the linear flow, and the nonlinear term is bounded by
$\Hp(T)\norm u^p$; this is where $\Hp$ comes from. The argument is in the
spirit of \cite{Weissler1981,QuittnerSouplet2019}.

\smallskip
\emph{Upper bound for the lifespan} (Theorem~\ref{thm:master}-(b)). We look at
the solution through a Gaussian window of width $R$ centered at the origin:
$F_R(t)=R^q[P_{\omega R^2-t}u(t)](0)$ is an average of $u(t)$ against a
Gaussian. By Jensen's inequality it satisfies an ordinary differential
inequality of the form $F_R'\ge h(t)R^{-q(p-1)}F_R^p$. On one time slab this is
the classical backward-semigroup argument of Weissler \cite{Weissler1981} and
Meier \cite{Meier1988} (see also \cite{QuittnerSouplet2019}); it is sharp off
the critical power, but it only bounds the contribution of a single scale and
therefore gives nothing at criticality. The test-function method of Mitidieri
and Pohozaev \cite{MitidieriPohozaev2001}, adapted to Grushin and Tricomi
operators in \cite{DAmbrosioLucente2003}, recovers the threshold $p\le1+2/q$
equally cheaply, and its refinement by Ikeda and Sobajima
\cite{IkedaSobajima2019} gives the sharp critical lifespan for $h\equiv1$, but
neither gives the law for general $h$. We add three things:
\begin{itemize}[nosep,leftmargin=1.6em]
\item a monotone Gaussian-envelope energy $E(R)$, which records the reaction
seen at all scales up to $R$;
\item a scale increment, saying that $E$ grows between two scales by an amount
proportional to the reaction in the corresponding slab, combined with an ODE
cap, saying that the reaction in one slab cannot be too large; together they
telescope along a geometric sequence of scales and sum the contributions of
\emph{all} scales;
\item a finite family of interlaced geometric grids of scales whose slabs cover
every time, which handles arbitrary $h$.
\end{itemize}

\subsection*{Organization of the paper}
Section~\ref{sec:prelim} collects the tools: the heat kernel of $\Lcal$, the
local theory of mild solutions, the regularity of classical solutions (the
proof of Theorem~\ref{thm:regularity}) and some elementary asymptotics.
Section~\ref{sec:proofs} proves Theorems~\ref{thm:master}, \ref{thm:threshold}
and~\ref{thm:apriori} and Corollary~\ref{cor:small}; the lower bound is in
Subsection~\ref{subsec:lower} and the upper bound, which is the heart of the
paper, in Subsection~\ref{subsec:upper}. Section~\ref{sec:applications} derives
the explicit laws, Theorems~\ref{thm:powerlog}--\ref{thm:nonstandard}, from
the master law. Section~\ref{sec:general} treats nonlinearities $f(u)$. A reader interested only in the results can stop after
Section~\ref{sec:intro}; a reader interested in the idea of the upper bound
should read ``Method of proof'' above and the plan at the beginning of
Subsection~\ref{subsec:upper}.

\section{Preliminaries}\label{sec:prelim}

This section collects the tools used in the proofs. Subsection~\ref{subsec:kernel}
recalls the heat kernel of $\Lcal$ and the properties of the semigroup $P_t$;
all of them are taken from the literature. Subsection~\ref{subsec:mild} gives
the local theory of mild solutions: uniqueness, short-time existence, the
blow-up alternative and comparison. These two subsections are all that is needed
for the lifespan law (Theorems~\ref{thm:master} and~\ref{thm:threshold}).
Subsection~\ref{subsec:classical} proves the regularity result
Theorem~\ref{thm:regularity} and shows that classical solutions are mild
solutions; it is used only for the results on classical solutions. Finally,
Subsection~\ref{subsec:asympt} contains elementary asymptotics used in
Section~\ref{sec:applications}.

\subsection{Homogeneous H\"ormander fields and the heat kernel}\label{subsec:kernel}

Assumption~\ref{ass:H} is the standing assumption of
\cite{BiagiBonfiglioli2017,BiagiBonfiglioli2024,BiagiBramanti2023,
ChatzakouKassymovRuzhansky2026}; (H1)--(H2) imply H\"ormander's condition at
every point. The coefficients $X_j^i$ are polynomials with $\partial_iX_j^i=0$
\cite{BiagiBonfiglioli2017}, so each $X_j$ is divergence free and $\Lcal$ is
formally self-adjoint in $L^2(\R^n)$. Let $\CC$ be the
Carnot--Carath\'eodory distance and $\BX(x,r)$ its balls. Then
$\CC(\delta_\lambda x,\delta_\lambda y)=\lambda\CC(x,y)$, and
\begin{equation}\label{eq:ball-origin}
 |\BX(0,r)|=r^q|\BX(0,1)|,\qquad r>0 .
\end{equation}
For other centers the volume is only bounded below by $cr^q$
\cite{BiagiBonfiglioliBramanti2022,BiagiBramanti2023}.

Let $\Gamma_t(x,y)$ be the global heat kernel of $\partial_t-\Lcal$
\cite{BiagiBonfiglioli2024}. We use the following facts, proved in
\cite{BiagiBonfiglioli2024,BiagiBramanti2023} and summarized in
\cite{ChatzakouKassymovRuzhansky2026}.
\begin{itemize}[leftmargin=1.6em]
\item \emph{Basic properties.} $\Gamma_t(x,y)=\Gamma_t(y,x)>0$,
$\int\Gamma_t(x,y)\dd y=1$, and
$\Gamma_{t+s}(x,y)=\int\Gamma_t(x,z)\Gamma_s(z,y)\dd z$. Moreover $\Gamma$ is
smooth for $t>0$ and solves the heat equation.
\item \emph{Gaussian bounds.} For some $\rho_G>1$,
\begin{equation}\label{eq:Gaussian-general}
 \frac{e^{-\rho_G\CC(x,y)^2/t}}{\rho_G|\BX(x,\sqrt t)|}
 \le\Gamma_t(x,y)\le
 \frac{\rho_G\,e^{-\CC(x,y)^2/(\rho_Gt)}}{|\BX(x,\sqrt t)|}.
\end{equation}
By \eqref{eq:ball-origin}, at the origin this becomes
\begin{equation}\label{eq:origin-gaussian}
 c_{\rm o}t^{-q/2}e^{-A_{\rm o}\CC(0,y)^2/t}\le\Gamma_t(0,y)\le
 C_{\rm o}t^{-q/2}e^{-a_{\rm o}\CC(0,y)^2/t}.
\end{equation}
\item \emph{The semigroup.} $P_tf=\int\Gamma_t(\cdot,y)f(y)\dd y$ is positivity
preserving and contractive on every $L^k$, $1\le k\le\infty$, conserves mass,
and is strongly continuous on $L^1$ (but not on $L^\infty$).
\item \emph{Decay.} By \cite{ChatzakouKassymovRuzhansky2026},
\begin{equation}\label{eq:L1Linfty}
 \norm{P_tf}_{k_2}\le Ct^{-\frac q2(\frac1{k_1}-\frac1{k_2})}\norm f_{k_1},
 \qquad1\le k_1\le k_2\le\infty .
\end{equation}
\end{itemize}
In other words, the heat flow of $\Lcal$ looks like the Euclidean heat flow, with
$n$ replaced by $q$ and the Euclidean distance replaced by $\CC$. This is all we
use about $\Lcal$.

\subsection{Mild solutions}\label{subsec:mild}

\begin{definition}\label{def:mild}
A mild solution on $[0,T)$ is a nonnegative $u\in C([0,T);L^1(\R^n))$,
bounded in $L^\infty$ on compact subintervals of $[0,T)$, with
\begin{equation}\label{eq:mild}
 u(t)=\eps P_tu_0+\int_0^th(s)P_{t-s}(u(s)^p)\dd s,\qquad 0\le t<T,
\end{equation}
the integral being $L^1$-valued. $T_\eps$ is the supremum of such $T$.
\end{definition}

Continuity is required only in $L^1$, because $P_t$ is not strongly continuous
on $L^\infty$. Since $\norm f_\infty$ is a supremum of continuous functionals
on $L^1$, we also have
$\norm{\int g\dd s}_\infty\le\int\norm g_\infty\dd s$ for $L^1$-Bochner
integrals; this is used throughout.

The following local theory is standard (Weissler \cite{Weissler1980},
\cite[Sect.~15]{QuittnerSouplet2019}); the only change is that $h\in L^1_{\rm loc}$
enters through $\Omega_h(\delta;\tau):=\sup_{s\le\tau}\int_s^{s+\delta}h\to0$ as
$\delta\to0$, which replaces the smallness of the time step.

\begin{lemma}\label{lem:local}
Assume \eqref{eq:h-ass}.
\begin{enumerate}[label=\textup{(\roman*)}]
\item Mild solutions are unique.
\item For $0\le v\in L^1\cap L^\infty$ and $K=\norm v_1+\norm v_\infty$, if
$p(2K)^{p-1}\Omega_h(\delta;t_0)\le\frac12$, there is a mild solution on
$[t_0,t_0+\delta]$ with initial value $v$ and $L^1,L^\infty$ norms $\le2K$. In
particular $T_\eps>0$.
\item If $T_\eps<\infty$, then $\limsup_{t\uparrow T_\eps}\norm{u(t)}_\infty=\infty$.
\end{enumerate}
\end{lemma}

\begin{inwords}
Solutions are unique; they exist for a short time depending only on the size of
the data; and a solution that exists on a finite maximal interval must become
unbounded at its end. So $T_\eps<\infty$ means $\norm{u(t)}_\infty\to\infty$
along some sequence $t\uparrow T_\eps$.
\end{inwords}

\begin{proof}[Sketch]
(i) Gronwall in $L^1$, using the Lipschitz bound of $\xi\mapsto\xi^p$ on bounded
sets. (ii) Banach's fixed point theorem on the set of nonnegative
$w\in C([t_0,t_0+\delta];L^1)$ with $\norm w_1,\norm w_\infty\le2K$, with metric
$\sup_t\norm{w_1-w_2}_1$; the smallness of $\delta$ makes the Duhamel map a
contraction. (iii) If $\norm{u(t)}_\infty\le B$ on $[0,T_\eps)$, Gronwall bounds
$\norm{u(t)}_1$, and (ii) restarts the solution at $T_\eps-\delta/2$ with
$\delta$ chosen uniformly, a contradiction.
\end{proof}

The next lemma says that the Duhamel integral is continuous in time with values
in $L^1$; it is needed to check that the iterates below stay in
$C([0,T);L^1)$.

\begin{lemma}\label{lem:duhamel-cont}
If $g:[0,\tau]\to L^1$ is strongly measurable and bounded in $L^1$, then
$t\mapsto\int_0^th(s)P_{t-s}g(s)\dd s$ lies in $C([0,\tau];L^1)$.
\end{lemma}

\begin{proof}
For $t<t'$ and $\Delta=t'-t$ the difference is
$\int_t^{t'}hP_{t'-s}g\dd s+\int_0^th(P_\Delta-I)P_{t-s}g\dd s$. The first term
is at most $\sup\norm g_1\int_t^{t'}h$. For the second, fix $s$: the set
$\{P_\sigma g(s):0\le\sigma\le\tau\}$ is compact in $L^1$ by strong continuity of
$P_\sigma$, so $\sup_{\sigma\le\tau}\norm{(P_\Delta-I)P_\sigma g(s)}_1\to0$ as
$\Delta\to0$, and this supremum is at most $2\sup\norm g_1$. Dominated
convergence with the weight $h$ shows that the second term tends to $0$
uniformly in $t$.
\end{proof}

Bigger data give bigger solutions and shorter lifespans, because $P_t$
preserves positivity and the nonlinearity $\xi\mapsto\xi^p$ is increasing.

\begin{lemma}[Comparison]\label{lem:comparison}
If $0<\eps_1\le\eps_2$, then $T_{\eps_1}\ge T_{\eps_2}$ and
$u_{\eps_1}\le u_{\eps_2}$ on $[0,T_{\eps_2})$.
\end{lemma}

\begin{proof}
Fix $\tau<T_{\eps_2}$. Define $v_0=0$ and
$v_{k+1}=\eps_1P_tu_0+\int_0^th\,P_{t-s}(v_k^p)\dd s$. By positivity of $P_t$,
$0\le v_k\le v_{k+1}\le u_{\eps_2}$. Each $v_k$ is in $C([0,\tau];L^1)$ by
Lemma~\ref{lem:duhamel-cont}, and $v=\sup_kv_k\le u_{\eps_2}$ solves
\eqref{eq:mild} with $\eps_1$ by monotone convergence, is bounded, and lies in
$C([0,\tau];L^1)$ again by Lemma~\ref{lem:duhamel-cont}. Hence
$T_{\eps_1}\ge\tau$, and $v=u_{\eps_1}$ by uniqueness.
\end{proof}

\subsection{Classical solutions: proof of Theorem~\ref{thm:regularity}}
\label{subsec:classical}

\begin{roadmap}
(1) We first prove derivative bounds (G) for the heat semigroup: on a Carnot
group they follow from scaling, and in general they are transferred to $\R^n$
through the global lifting. (2) Lemma~\ref{lem:holder-smoothing} shows that the
heat flow moves points only a distance $\tau^{1/2}$ and smooths H\"older
functions. (3) The proof of Theorem~\ref{thm:regularity} then proceeds in four
parts: (a) positivity, continuity and the distributional equation; (b)
smoothness for smooth $h$, by hypoellipticity; (c) the Schauder-type estimates,
in four steps (an a priori Lipschitz bound, regularity of $X_iX_ju$ in space,
regularity in time, and the time derivative); (d) the conclusion. (4)
Lemma~\ref{lem:classical-mild} shows that every classical solution is a mild
solution, which links this subsection to the lifespan theorems.
\end{roadmap}

Let
\[
 [g]_\theta=\sup_{x\ne y}|g(x)-g(y)|/\CC(x,y)^\theta .
\]
Besides the Gaussian
bounds we need derivative estimates for the heat semigroup. For ordered
products $X^I=X_{i_1}\cdots X_{i_k}$ of the fields:

\smallskip
\noindent\textbf{(G)} \ $P_\tau g\in C^\infty(\R^n)$ for $g\in L^\infty$, $\tau>0$, and for every $k\ge1$
\[
 \norm{X^IP_\tau g}_\infty\le C_k\,\tau^{-k/2}\norm g_\infty,
 \qquad
 \norm{X^IP_\tau g}_\infty\le C_k\,\tau^{-(k-\varkappa)/2}[g]_\varkappa\quad(0<\varkappa\le1),
\]
for all $|I|=k$, the second bound for $g\in L^\infty$ with $[g]_\varkappa<\infty$.

\smallskip

\begin{idea}
On a Carnot group the heat semigroup is a convolution with a kernel that obeys
the scaling $\gamma_\tau(z)=\tau^{-Q/2}\gamma_1(\delta_{\tau^{-1/2}}z)$, so every
derivative along the fields costs exactly a factor $\tau^{-1/2}$. For a general
$\Lcal$ we lift to a Carnot group, apply the Carnot case, and project back.
\end{idea}

On a Carnot group $P_\tau g=g*\gamma_\tau$, $\gamma_\tau(z)=\tau^{-Q/2}\gamma_1(\delta_{\tau^{-1/2}}z)$
\textup{(}$Q$ the homogeneous dimension of the group\textup{)}
with $\gamma_1$ Schwartz \cite{Folland1975,VaropoulosSaloffCosteCoulhon1992},
and $X^I(g*\gamma_\tau)=g*(X^I\gamma_\tau)$ by left invariance. Since
$X^I\gamma_\tau(z)=\tau^{-Q/2-k/2}(X^I\gamma_1)(\delta_{\tau^{-1/2}}z)$,
\begin{equation}\label{eq:kernel-scaling}
 \int|X^I\gamma_\tau|\dd z=\tau^{-k/2}\!\int|X^I\gamma_1|,
 \qquad
 \int|X^I\gamma_\tau(z)|\,|z|^\varkappa\dd z=\tau^{-(k-\varkappa)/2}\!\int|X^I\gamma_1(z)||z|^\varkappa\dd z,
\end{equation}
where $|z|=d(0,z)$ is the (left-invariant, homogeneous) Carnot--Carath\'eodory
norm; both integrals are finite. Moreover $\int X^I\gamma_\tau\dd z=0$ for
$k\ge1$, because $\gamma_\tau$ is Schwartz and the fields are divergence free.
Hence $X^I(g*\gamma_\tau)(w)=\int(g(v)-g(w))X^I\gamma_\tau(v^{-1}w)\dd v$, and
$|g(v)-g(w)|\le[g]_\varkappa|v^{-1}w|^\varkappa$ gives the second bound, the first
being the first integral in \eqref{eq:kernel-scaling}.

In general, let $G=(\R^{n+p'},\star)$ be the Carnot group of the global lifting,
with lifted fields $Z_j=X_j+R_j$, where $R_j$ differentiates only in the extra
variables, and $\pi(x,\xi)=x$. Then $Z_j(F\circ\pi)=(X_jF)\circ\pi$, and since
$\pi_*Z_j=X_j$, $\pi$ does not increase Carnot--Carath\'eodory distances. By
\cite[Thm.~1.4]{BiagiBonfiglioli2024} the heat kernel of $\Lcal$ is the
saturation $\Gamma_t(x,y)=\int\gamma^G_t((x,0)^{-1}\star(y,\eta))\dd\eta$ of the
heat kernel $\gamma^G_t$ of $\sum Z_j^2$, so
$P_tg(x)=\widetilde P_t(g\circ\pi)(x,0)$ with $\widetilde P_t$ the heat
semigroup of $G$. The function $\widetilde P_t(g\circ\pi)$ is independent of the
extra variables (by uniqueness of bounded solutions for continuous $g$, then
by a monotone-class argument), so
$(P_tg)\circ\pi=\widetilde P_t(g\circ\pi)$ is $C^\infty$ and
$(X^IP_tg)\circ\pi=Z^I\widetilde P_t(g\circ\pi)$. The Carnot case applied to
$\tilde g=g\circ\pi$, together with
$|\tilde g(v)-\tilde g(w)|\le[g]_\varkappa\CC(\pi v,\pi w)^\varkappa\le[g]_\varkappa|v^{-1}w|^\varkappa$
for $w=(x,0)$, gives (G) with $C_k$ depending only on $G$ and $k$. Only the
saturation formula and the lifting identities are taken from
\cite{BiagiBonfiglioli2024}; no pointwise derivative bound for $\Gamma$ is
needed. The estimates (G) are used only in Theorem~\ref{thm:regularity}-(c).

For $g\in C^2_X$ (two continuous $X$-derivatives) such that $g$, $X_jg$ and
$X_iX_jg$ are bounded, we also have
\begin{equation}\label{eq:P-Duhamel}
 P_\tau g-g=\int_0^\tau P_\rho(\Lcal g)\dd\rho .
\end{equation}
Indeed, $\widetilde P_\rho$ is self-adjoint on $G$, so
$\partial_\rho\widetilde P_\rho\tilde g=\widetilde P_\rho(\sum Z_j^2\tilde g)$
by testing against $\varphi\in C_c^\infty(G)$ and integrating by parts, and
$\widetilde P_\rho\tilde g\to\tilde g$ locally uniformly as $\rho\to0$;
restrict to $\xi=0$.

\medskip

\begin{lemma}\label{lem:holder-smoothing}
Let $\theta\in(0,1]$. There are $C_D,C_H,C_S$ such that
\begin{enumerate}[label=\textup{(\roman*)}]
\item $|\BX(x,2r)|\le C_D|\BX(x,r)|$ for all $x,r$;
\item $\int\Gamma_\tau(x,y)\CC(x,y)^\theta\dd y\le C_H\tau^{\theta/2}$, hence
$\norm{P_\tau g-g}_\infty\le C_H\tau^{\theta/2}[g]_\theta$;
\item $[P_\tau g]_\theta\le C_S\tau^{-\theta/2}\norm g_\infty$.
\end{enumerate}
\end{lemma}

\begin{inwords}
(i) Balls of radius $2r$ have at most a fixed multiple of the volume of balls of
radius $r$. (ii) The heat flow moves mass a distance of order $\tau^{1/2}$, so
$P_\tau g$ is close to $g$ when $g$ is H\"older. (iii) Conversely, $P_\tau$
smooths: it turns a bounded function into a H\"older one, with a loss
$\tau^{-\theta/2}$.
\end{inwords}

\begin{proof}
(i) By \eqref{eq:Gaussian-general} with $t=r^2$, $\Gamma_t(x,y)\ge
e^{-4\rho_G}/(\rho_G|\BX(x,r)|)$ on $\BX(x,2r)$; since $\Gamma_t(x,\cdot)$ has
mass one, $C_D=\rho_Ge^{4\rho_G}$ works.
(ii) Split into $\{\CC<\sqrt\tau\}$ and dyadic annuli
$A_k=\{2^{k-1}\sqrt\tau\le\CC<2^k\sqrt\tau\}$. By the upper bound in
\eqref{eq:Gaussian-general} and (i), the integral over $A_k$ is at most
$\rho_G(2^k\sqrt\tau)^\theta C_D^ke^{-4^{k-1}/\rho_G}$, which is summable.
(iii) By Chow's theorem and (G), $|P_\tau g(x)-P_\tau g(y)|\le
\sqrt mC_1\tau^{-1/2}\CC(x,y)\norm g_\infty$ along a sub-unit curve; together with
the bound $2\norm g_\infty$ this gives $C\min\{1,\CC/\sqrt\tau\}\norm g_\infty
\le C(\CC/\sqrt\tau)^\theta\norm g_\infty$.
\end{proof}

\begin{proof}[Proof of Theorem~\ref{thm:regularity}]
(a) Positivity of $h,u$ and $P_t$ gives $u\ge\eps P_tu_0$, which is continuous
and positive for $t>0$ by \eqref{eq:Gaussian-general}. Write
$u-\eps P_tu_0=\int_0^th(s)\int\Gamma_{t-s}(x,y)u(s,y)^p\dd y\dd s$ with
$0\le u^p\le B^p$. The part of the $s$-integral over $[t-\delta,t]$ is at most
$B^p\int_{t-\delta}^th$, uniformly small, and on $[0,t-\delta]$ the integrand is
continuous in $(t,x)$ by Scheff\'e's lemma and dominated by $B^ph(s)$.

\emph{Distributional equation.} For $\psi\in C_c^\infty$,
$\partial_rP_r\psi=P_r\Lcal\psi$, and $P_r\psi\to\psi$ locally uniformly, so
$P_t\psi=\psi+\int_0^tP_r\Lcal\psi\dd r$. Pair \eqref{eq:mild} with $\Lcal\psi$,
integrate over $\tau\in[0,t]$ and use Fubini (justified by $h\in L^1_{\rm loc}$
and $|\langle u^p,P_r\Lcal\psi\rangle|
\le\norm u_\infty^{p-1}\norm u_1\norm{\Lcal\psi}_\infty$) to get
\begin{equation}\label{eq:weak-form}
 \langle u(t),\psi\rangle=\eps\langle u_0,\psi\rangle
 +\int_0^t\bigl[\langle u,\Lcal\psi\rangle+h\langle u^p,\psi\rangle\bigr]\dd\tau ;
\end{equation}
the equation $u_t-\Lcal u=hu^p$ holds in the sense of distributions on
$(0,T_\eps)\times\R^n$, by testing with $\chi(t)\psi(x)$, $\chi\in C_c^\infty(0,T_\eps)$.

(b) $P=\partial_t-\Lcal$ is hypoelliptic with the subelliptic gain
$H^s_{\rm loc}\to H^{s+\sigma_0}_{\rm loc}$ \cite{Hormander1967,
RothschildStein1976}. On compacts of $I\times\R^n$, $c\le u\le M$ by (a), so
$u^p=G(u)$ with $G\in C^\infty$. Starting from $u\in L^\infty_{\rm loc}$,
Moser's composition estimate \cite{Taylor2011} gives
$Pu=hG(u)\in H^s_{\rm loc}\Rightarrow u\in H^{s+\sigma_0}_{\rm loc}$, and
induction plus Sobolev embedding yield $u\in C^\infty$.

\smallskip\noindent\emph{Idea of (c).} The Duhamel integral is split into three
parts: a far-past part, which is smooth because the semigroup has had time at
least $\delta$ to smooth it; a near-past part in which $F(s)$ is compared with
$F(t)$ (this is where the H\"older continuity of $h$ in time is used); and a
``frozen'' part with $F(t)$ in place of $F(s)$ (this is where the H\"older
continuity in space is used). Each part is bounded with the derivative
estimates (G), and the singular factors $(t-s)^{-1+\varkappa/2}$ are integrable
because $\varkappa>0$.

(c) We may take $\varkappa<1$ with $\varkappa\le2\gamma$. Choose compact
intervals $J\subset J_0\subset J_1\subset I$ and $\delta>0$ with
$[t-\delta,t]\subset J_0$ for $t\in J$, and put $t_J=\min J>0$,
$B=\sup_{[0,\max J_1]}\norm u_\infty$, $F:=h\,u^p$.

\emph{Step 0: an a priori Lipschitz bound.} We first show that $u$ is Lipschitz
in $x$ and $\frac12$-H\"older in $t$, which gives the bounds \eqref{eq:F-hold}
on $F=hu^p$ needed below. Let $\delta_0>0$ be so small that
$J_2:=\{t\in\R:\operatorname{dist}(t,J_1)\le\delta_0\}\subset I$, and let
$M_h=\sup_{J_2}h<\infty$. By Lemma~\ref{lem:holder-smoothing}-(iii) and
\eqref{eq:mild}, for $t\in J_1$,
\[
 [u(t)]_1\le C_S\eps\,(\min J_1)^{-1/2}\norm{u_0}_\infty
 +C_SB^p\int_0^th(s)(t-s)^{-1/2}\dd s .
\]
Over $[0,t-\delta_0]$ the integral is at most
$\delta_0^{-1/2}\int_0^{\max J_1}h$, and over $[t-\delta_0,t]$ it is at most
$2M_h\delta_0^{1/2}$. So $K:=\sup_{J_1}[u(t)]_1<\infty$. For $t<t'$ in $J_1$
and $\tau=t'-t$ the semigroup law gives
$u(t')=P_\tau u(t)+\int_t^{t'}hP_{t'-s}u^p\dd s$, so by
Lemma~\ref{lem:holder-smoothing}-(ii)
\[
 |u(t',x)-u(t,x)|\le C_HK\tau^{1/2}+B^pM_h\tau\le C\tau^{1/2}.
\]
Thus $u$ is Lipschitz in $x$ with respect to $\CC$ and $\frac12$-H\"older in $t$
on $J_1\times\R^n$. Since $s\mapsto e^{sX_k}x$ is sub-unit, $|X_ku|\le K$
wherever $X_ku$ exists.

Since $|h(t)-h(s)|\le[h]_\gamma|t-s|^\gamma$ and
$\varkappa/2\le\min\{\gamma,\frac12\}$, there is $M$ with, for $t,s\in J_0$,
\begin{equation}\label{eq:F-hold}
 \norm{F(t)}_\infty\le M,\qquad [F(t)]_\varkappa\le M,\qquad
 \norm{F(t)-F(s)}_\infty\le M|t-s|^{\varkappa/2}
\end{equation}
(for $\CC(x,y)>1$ the Lipschitz bound is replaced by $2M\le2M\CC^\varkappa$).
Write $D=X_iX_j$ for a fixed ordered pair. We use freely that derivatives along the
fields $X_k$ may be taken under the integral sign when the differentiated
integrand is continuous and dominated by an integrable function of the
integration variable (integrate along the integral curves of $X_k$ and use
dominated convergence). We also write $|\psi(x)-\psi(y)|\le\sqrt m\,
\max_k\norm{X_k\psi}_\infty\CC(x,y)$ (Chow's theorem), used with $\psi=X^IP_\sigma g$.

\emph{Step 1: H\"older regularity of $X_iX_ju$ in space.} We follow a classical approach, well known in $\R^n$ (see, for example,
\cite{QuittnerSouplet2019}).

Fix $t\in J$. By
\eqref{eq:mild},
$u(t)=\eps P_tu_0+w_1+w_2+w_3$ with
\[
 w_1=\int_0^{t-\delta}\!\!hP_{t-s}u^p\dd s,\quad
 w_2=\int_{t-\delta}^t\!\!P_{t-s}\bigl(F(s)-F(t)\bigr)\dd s,\quad
 w_3=\int_0^\delta P_\sigma F(t)\dd\sigma .
\]
By (G), $\norm{DP_tu_0}_\infty\le C_2t_J^{-1}\norm{u_0}_\infty$ and
$\norm{X_kDP_tu_0}_\infty\le C_3t_J^{-3/2}\norm{u_0}_\infty$. For $w_1$,
$\norm{DP_{t-s}u(s)^p}_\infty\le C_2\delta^{-1}B^p$ and
$\norm{X_kDP_{t-s}u(s)^p}_\infty\le C_3\delta^{-3/2}B^p$, integrated against
$h\in L^1(0,\max J)$. For $w_2$, by \eqref{eq:F-hold} and the first bound in (G),
$\norm{DP_{t-s}(F(s)-F(t))}_\infty\le C_2M(t-s)^{-1+\varkappa/2}$, and the same with
$X_kD$, $C_3$ and $(t-s)^{-3/2+\varkappa/2}$. For $w_3$, the second bound in (G) with
$[F(t)]_\varkappa\le M$ gives $\norm{DP_\sigma F(t)}_\infty\le C_2M\sigma^{-1+\varkappa/2}$
and $\norm{X_kDP_\sigma F(t)}_\infty\le C_3M\sigma^{-3/2+\varkappa/2}$. The
dominating functions for $D$, of order $\sigma^{-1+\varkappa/2}$, are integrable
at the singular endpoint because $\varkappa>0$, so $Du(t)$ exists, is continuous
and bounded. The bounds for $X_kD$, of order $\sigma^{-3/2+\varkappa/2}$, are not
integrable at $\sigma=0$; they are used below only for $\sigma>d^2$. For $x,y$ with $d=\CC(x,y)<1$, split each of
$w_2,w_3$ at $\sigma=d^2$: using the $C_2$-bound for $\sigma\le d^2$ and the
$C_3$-bound for $\sigma>d^2$,
\[
 |Dw_{2,3}(x)-Dw_{2,3}(y)|\le CM\Bigl(\int_0^{d^2}\!\sigma^{-1+\frac\varkappa2}\dd\sigma
 +d\int_{d^2}^\infty\!\sigma^{-\frac32+\frac\varkappa2}\dd\sigma\Bigr)
 \le CM\Bigl(\tfrac2\varkappa+\tfrac2{1-\varkappa}\Bigr)d^\varkappa .
\]
The terms $P_tu_0$ and $w_1$ are Lipschitz, hence $\varkappa$-H\"older for $d<1$, and
bounded for $d\ge1$. Thus $\sup_{t\in J}[X_iX_ju(t)]_\varkappa<\infty$. The same
bounds with $|I|=1$ in (G) show that $X_ku(t)$ exists; it is bounded by $K$.

\emph{Step 2: H\"older regularity of $X_iX_ju$ in time.} Let $t<t'$ in $J$,
$\tau=t'-t$, $g=u(t)$. By \eqref{eq:mild}
and the semigroup law, $u(t')=P_\tau g+W$ with $W=\int_t^{t'}P_{t'-s}F(s)\dd s$.
By Steps 0 and 1, $g\in C^2_X$ with $g$, $X_kg$, $X_iX_kg$ bounded, and
$[\Lcal g]_\varkappa\le C$ for $\Lcal g=\sum_kX_k^2g$. By \eqref{eq:P-Duhamel} and (G),
\[
 \norm{D(P_\tau g-g)}_\infty\le\int_0^\tau\norm{DP_\rho\Lcal g}_\infty\dd\rho
 \le C_2[\Lcal g]_\varkappa\int_0^\tau\rho^{-1+\frac\varkappa2}\dd\rho\le C\tau^{\varkappa/2}.
\]
For $W$ write $F(s)=F(t')+(F(s)-F(t'))$ and use (G) and \eqref{eq:F-hold}:
$\norm{DP_{t'-s}F(t')}_\infty\le C_2M(t'-s)^{-1+\varkappa/2}$ and
$\norm{DP_{t'-s}(F(s)-F(t'))}_\infty\le C_2M(t'-s)^{-1+\varkappa/2}$, so
$\norm{DW}_\infty\le C\tau^{\varkappa/2}$. Hence
$|Du(t',x)-Du(t,x)|\le C\tau^{\varkappa/2}$. With Step 1 this gives
\eqref{eq:schauder} for $D=X_iX_j$, and in particular $Du$ is continuous. The
same argument with $D$ replaced by $X_k$ and the first bound in (G) with $|I|=1$
gives $\norm{X_k(P_\tau g-g)}_\infty\le C_1\norm{\Lcal g}_\infty
\int_0^\tau\rho^{-1/2}\dd\rho$ and $\norm{X_kW}_\infty\le C\tau^{1/2}$, hence
$|X_ku(t',x)-X_ku(t,x)|\le C\tau^{1/2}$; with the Lipschitz bound in $x$,
$X_ku$ is continuous on $J\times\R^n$.

\emph{Step 3: the time derivative.} By \eqref{eq:weak-form} and $u\in C^2_X$
with $X_k$ divergence free,
$\frac{\dd}{\dd t}\langle u(t),\psi\rangle=\langle\Lcal u+F,\psi\rangle$ with a
continuous right side, for every $\psi\in C_c^\infty$. Hence
$u(t',x)-u(t,x)=\int_t^{t'}(\Lcal u+F)(s,x)\dd s$ for all $x$, so $u_t=\Lcal u+F$
exists and is continuous. It is the sum of the H\"older functions $X_k^2u$ and
$F$, which gives \eqref{eq:schauder} for $D=u_t$, using \eqref{eq:F-hold} and
Step 2.

(d) Apply (c) with $I=(0,T_\eps)$ and exhaust $(0,T_\eps)$ by compact
intervals: $X_ku$, $X_iX_ju$ and $u_t$ exist and are continuous on
$(0,T_\eps)\times\R^n$, bounded on $[t_0,t_1]\times\R^n$ for
$0<t_0<t_1<T_\eps$, and $u_t-\Lcal u=hu^p$ pointwise. The remaining requirements
\textup{(}$u\ge0$, $u\in C([0,T_\eps);L^1)$, $u(0)=\eps u_0$, boundedness on
$[0,t]\times\R^n$\textup{)} are part of the definition of a mild solution, so
$u$ is classical. The converse statement is
Lemma~\ref{lem:classical-mild}, and the blow-up alternative is
Lemma~\ref{lem:local}-(iii).
\end{proof}

\begin{lemma}[Classical solutions are mild]\label{lem:classical-mild}
Assume \eqref{eq:h-ass} and $h\in C((0,\infty))$. Every classical solution on
$[0,T)$ is a mild solution on $[0,T)$. Consequently $T\le T_\eps$, and it
coincides with the mild solution $u$ on $[0,T)$.
\end{lemma}

\begin{idea}
Fix $t$ and $x$ and differentiate $s\mapsto[P_{t-s}w(s)](x)$. The heat equation
for $w$ cancels the derivative of the semigroup and leaves
$h(s)P_{t-s}(w(s)^p)(x)$; integrating from $0$ to $t$ gives the Duhamel
formula.
\end{idea}

\begin{proof}
Let $w$ be a classical solution on $[0,T)$, fix $0<t<T$ and $x\in\R^n$, and put
$\phi(s)=[P_{t-s}w(s)](x)$ for $0<s<t$. Fix $0<s_0<t<t_1<T$ and let $B$ bound
$|w|$, $|X_kw|$, $|X_iX_jw|$ and $|w_t|$ on $[s_0,t_1]\times\R^n$. For
$s\in(s_0,t)$,
\[
 \phi(s+\delta)-\phi(s)=P_{t-s-\delta}\bigl[w(s+\delta)-w(s)\bigr](x)
 +\bigl[P_{t-s-\delta}w(s)-P_{t-s}w(s)\bigr](x).
\]
Since $\delta^{-1}[w(s+\delta)-w(s)]\to w_t(s)$ pointwise with bound $B$, and
$\Gamma_\tau(x,\cdot)\to\Gamma_{t-s}(x,\cdot)$ in $L^1$ as $\tau\to t-s$
\textup{(}Scheff\'e\textup{)}, the first term divided by $\delta$ tends to
$P_{t-s}w_t(s)(x)$. For $g=w(s)$, which lies in $C^2_X$ with $g$, $X_jg$,
$X_iX_jg$ bounded, \eqref{eq:P-Duhamel} shows that $\tau\mapsto P_\tau g(x)$ is
$C^1$ with derivative $P_\tau(\Lcal g)(x)$, so the second term divided by
$\delta$ tends to $-P_{t-s}(\Lcal w(s))(x)$. Hence $\phi$ is differentiable on
$(0,t)$ with
$\phi'(s)=P_{t-s}(w_t-\Lcal w)(s)(x)=h(s)P_{t-s}\bigl(w(s)^p\bigr)(x)$, which is
continuous \textup{(}Scheff\'e's lemma and dominated convergence\textup{)}, and for $s_0<s_1<t$
\[
 \phi(s_1)-\phi(s_0)=\int_{s_0}^{s_1}h(s)P_{t-s}\bigl(w(s)^p\bigr)(x)\dd s .
\]
As $s_1\uparrow t$, $\phi(s_1)\to w(t,x)$: indeed
$\sup_y|w(s_1,y)-w(t,y)|\le B(t-s_1)$, and
$|P_\tau w(t)(x)-w(t,x)|\le C_H\tau^{1/2}[w(t)]_1\to0$ by
Lemma~\ref{lem:holder-smoothing}-(ii), because $w(t)$ is Lipschitz by Chow's
theorem. As $s_0\downarrow0$, $\phi(s_0)\to\eps P_tu_0(x)$: here
$w(s_0)\to\eps u_0$ in $L^1$, $\sup_y\Gamma_\tau(x,y)\le C\tau^{-q/2}$ for
$\tau\in[t/2,t]$ by \eqref{eq:Gaussian-general} and the volume bound, and
$\Gamma_{t-s_0}(x,\cdot)\to\Gamma_t(x,\cdot)$ in $L^1$ while $u_0\in L^\infty$. The
integrand is nonnegative, so by monotone convergence
\[
 w(t,x)=\eps P_tu_0(x)+\int_0^th(s)P_{t-s}\bigl(w(s)^p\bigr)(x)\dd s
\]
for every $x$. Finally, $s\mapsto w(s)^p\in L^1$ is continuous, because
$|a^p-b^p|\le pB'^{p-1}|a-b|$ for $0\le a,b\le B'$ with $B'$ a bound of $w$ on
$[0,t]\times\R^n$, and
$\int_0^th\norm{P_{t-s}w(s)^p}_1\dd s\le B'^{p-1}\int_0^th\norm{w(s)}_1\dd s<\infty$.
So the integral is an $L^1$-valued Bochner integral whose value agrees almost
everywhere with the pointwise integral \textup{(}Fubini\textup{)}, and $w$ is a
mild solution. By the definition of $T_\eps$, $T\le T_\eps$, and $w=u$ on
$[0,T)$ by Lemma~\ref{lem:local}-(i).
\end{proof}

\subsection{Elementary asymptotics}\label{subsec:asympt}

The following lemma is used in Section~\ref{sec:applications} to compute
$\Hp$ for explicit coefficients and to invert the law \eqref{eq:master}.

\begin{lemma}\label{lem:asympt}
\begin{enumerate}[label=\textup{(\roman*)}]
\item Let $\rho>0$, $\gamma\in\R$, $T_0\ge e$. Then
$\int_{T_0}^Tt^{\rho-1}(\log t)^\gamma\dd t\sim\rho^{-1}T^\rho(\log T)^\gamma$.
If $\rho<0$ the integral converges as $T\to\infty$.
\item Let $\rho>0$, $\gamma\in\R$ and $G(T)=T^\rho(\log T)^\gamma$. Let
$\Sigma>e$ be a large parameter \textup{(}in applications $\Sigma=\eps^{-(p-1)}$\textup{)}
and let $T=T(\Sigma)\to\infty$ as $\Sigma\to\infty$ satisfy
$c_1\Sigma\le G(T)\le C_1\Sigma$ for some constants $c_1,C_1>0$. Then
\[
 T\asymp\Sigma^{1/\rho}(\log\Sigma)^{-\gamma/\rho}\qquad(\Sigma\to\infty).
\]
\item For $\theta>0$, $\nu\in\R$:
$\int_{T_0}^Te^{\theta t}t^\nu\dd t\sim\theta^{-1}e^{\theta T}T^\nu$.
\end{enumerate}
\end{lemma}

\begin{inwords}
Parts (i) and (iii) say that an integral of a power times a logarithm, or of an
exponential times a power, is asymptotic to its leading term. Part (ii) says
that such a relation can be inverted: if $G(T)\asymp\Sigma$ then $T$ is of the
order of the inverse function of $G$ evaluated at $\Sigma$.
\end{inwords}

\begin{proof}
(i) and (iii) follow from l'H\^opital's rule, since
$\frac{\dd}{\dd T}[T^\rho(\log T)^\gamma]=T^{\rho-1}(\log T)^\gamma(\rho+\gamma/\log T)$
and $\frac{\dd}{\dd T}[\theta^{-1}e^{\theta T}T^\nu]=e^{\theta T}T^\nu
(1+\nu/(\theta T))$. (ii) For $K\ge1$ and $T\ge e$,
$G(KT)/G(T)\ge\phi(K):=K^\rho(1+\log K)^{-|\gamma|}\to\infty$ as $K\to\infty$.
Put $S=\Sigma^{1/\rho}(\log\Sigma)^{-\gamma/\rho}$; then $c_2\Sigma\le G(S)\le C_2\Sigma$
for large $\Sigma$, since $\log S/\log\Sigma\to1/\rho$. Fix $K$ with
$\phi(K)>\max\{C_1/c_2,C_2/c_1\}$. If $T\ge KS$, then $G(T)\ge G(KS)\ge\phi(K)G(S)>C_1\Sigma$,
and if $S\ge KT$, then $G(S)\ge G(KT)\ge\phi(K)G(T)>C_2\Sigma$, for large $\Sigma$
\textup{(}using that $G$ is increasing for large arguments\textup{)}; both contradict
the assumptions. So $S/K<T<KS$.
\end{proof}

\section{Proofs of the main results}\label{sec:proofs}

This section contains the proofs of Theorems~\ref{thm:master},
\ref{thm:threshold} and~\ref{thm:apriori} and of Corollary~\ref{cor:small}.
Part (a) of Theorem~\ref{thm:master} (the lower bound for the lifespan) is a
fixed-point argument; part (b) (the upper bound) is the main new argument of the
paper.

\subsection{Proof of Theorem~\ref{thm:master}-(a)}\label{subsec:lower}

\begin{roadmap}
We show that the Duhamel map is a contraction on a ball of a Banach space whose
norm is adapted to the decay of the linear flow. The steps are: the linear part
(Step 1); the $L^1$ bound of the nonlinear term (Step 2); the weighted $L^\infty$
bound, where the time integral is split at $t/2$ (Step 3); and the contraction
argument (Step 4). The coefficient $h$ enters only through the single number
$\Hp(T)$.
\end{roadmap}

\smallskip\noindent\textbf{Setting.}
Let $\mathcal X_T$ be the Banach space of $u\in C([0,T);L^1)$ with
\[
 \norm u_{\mathcal X_T}:=\sup_{0\le t<T}\bigl\{\norm{u(t)}_1+(1+t)^{q/2}\norm{u(t)}_\infty\bigr\}<\infty,
\]
and let $\Phi u(t)=\eps P_tu_0+\mathcal N(u)(t)$, where
$\mathcal N(u)(t)=\int_0^th(s)P_{t-s}(u(s)^p)\dd s$. Put
$\mathcal U_0=\norm{u_0}_1+\norm{u_0}_\infty$.

\smallskip\noindent\textbf{Step 1: the linear part.} For $t\le1$,
contractivity gives $\norm{P_tu_0}_1\le\norm{u_0}_1$ and
$(1+t)^{q/2}\norm{P_tu_0}_\infty\le2^{q/2}\norm{u_0}_\infty$. For $t\ge1$,
\eqref{eq:L1Linfty} gives $\norm{P_tu_0}_\infty\le Ct^{-q/2}\norm{u_0}_1\le
C'(1+t)^{-q/2}\norm{u_0}_1$. Hence $\norm{P_\cdot u_0}_{\mathcal X_T}\le C_{\rm L}\mathcal U_0$.

\smallskip\noindent\textbf{Step 2: the $L^1$ norm of the nonlinear term.} For
$u\in\mathcal X_T$ we have $\norm{u^p}_1\le\norm u_\infty^{p-1}\norm u_1$, so
\begin{equation}\label{eq:up-bounds}
 \norm{u(s)^p}_1\le(1+s)^{-\alpha}\norm u_{\mathcal X_T}^p,\qquad
 \norm{u(s)^p}_\infty\le(1+s)^{-qp/2}\norm u_{\mathcal X_T}^p .
\end{equation}
Since $P_{t-s}$ is a contraction on $L^1$, this gives
$\norm{\mathcal N(u)(t)}_1\le\Hp(T)\norm u_{\mathcal X_T}^p$. This is the point
where $\Hp$ appears.

\smallskip\noindent\textbf{Step 3: the weighted $L^\infty$ norm of the nonlinear
term.} Take $t>2$ and split $\mathcal N(u)(t)$ at $t/2$. On $[0,t/2]$ we have
$t-s\ge t/2$, so by \eqref{eq:L1Linfty} and \eqref{eq:up-bounds} this part is at
most $Ct^{-q/2}\Hp(T)\norm u^p_{\mathcal X_T}$. On $[t/2,t]$ we use
$L^\infty$-contractivity and $qp/2=q/2+\alpha$, which bound this part by
\[
 \norm u_{\mathcal X_T}^p\int_{t/2}^th(s)(1+s)^{-q/2-\alpha}\dd s\le
 2^{q/2}(1+t)^{-q/2}\Hp(T)\norm u^p_{\mathcal X_T}.
\]
For $t\le2$ one uses $L^\infty$-contractivity on all of $[0,t]$ in the same way.

\smallskip\noindent\textbf{Step 4: contraction and conclusion.} Steps 1--3 give
\[
 \norm{\Phi u}_{\mathcal X_T}\le C_{\rm L}\eps\mathcal U_0+C_{\rm N}\Hp(T)\norm u^p_{\mathcal X_T},
\]
and, using $|\xi^p-\zeta^p|\le p(\xi^{p-1}+\zeta^{p-1})|\xi-\zeta|$ in the same
computation,
\[
 \norm{\Phi u-\Phi v}_{\mathcal X_T}\le C_{\rm Lip}\Hp(T)\bigl(\norm u^{p-1}_{\mathcal X_T}
 +\norm v^{p-1}_{\mathcal X_T}\bigr)\norm{u-v}_{\mathcal X_T}.
\]
Let $M_\eps=2C_{\rm L}\eps\mathcal U_0$ and
$c_{\rm fp}=[4\max\{C_{\rm N},2C_{\rm Lip}\}(2C_{\rm L}\mathcal U_0)^{p-1}]^{-1}$.
If $\eps^{p-1}\Hp(T)\le c_{\rm fp}$, then $\Phi$ maps the set of nonnegative
$u$ with $\norm u_{\mathcal X_T}\le M_\eps$ into itself and is a
$\frac14$-contraction there. Its fixed point is a mild solution on $[0,T)$
satisfying \eqref{eq:X-bound} with $C=2C_{\rm L}\mathcal U_0$, and equals the
maximal solution by Lemma~\ref{lem:local}-(i). This proves (a).

\subsection{Proof of Theorem~\ref{thm:master}-(b)}\label{subsec:upper}

\begin{roadmap}
\emph{Step 1} introduces a functional $F_R$ that looks at the solution through a
Gaussian window of width $R$, and shows that it obeys the differential
inequality $F_R'\ge hR^{-q(p-1)}F_R^p$. \emph{Step 2} introduces an energy
$E(R)$ that records all the reaction seen at scales up to $R$, and bounds $F_R$
from below by it. \emph{Step 3} turns this into two inequalities for the amount
of reaction $\Lambda(R)$ in a time slab: an \emph{increment} inequality
(the energy must grow by at least $c\Lambda(\eps+E)^p$) and a \emph{cap}
($\Lambda(\eps+E)^{p-1}$ is bounded). \emph{Step 4} telescopes these two
inequalities along a geometric sequence of scales and sums the reaction over
\emph{all} scales: $\sum_jW_j\le C\eps^{-(p-1)}$. \emph{Step 5} uses finitely
many interlaced geometric grids, whose slabs cover every time, so that this sum
controls $\int h\,t^{-\alpha}\dd t$ for an arbitrary $h$.
\end{roadmap}

Fix numerical parameters $0<r^2<a<b<\omega$ (with $r$ decreased later). Let $u$
be a mild solution on $[0,T)$, with $\eps>0$ arbitrary. For $R>0$ and
$0\le t\le bR^2<T$ set
\[
 F_R(t):=R^q[P_{\omega R^2-t}u(t)](0)=\int u(t,x)K_R(t,x)\dd x,
 \qquad K_R(t,x):=R^q\Gamma_{\omega R^2-t}(0,x),
\]
so that $\int K_R(t,x)\dd x=R^q$. Thus $F_R(t)$ is the average of $u(t)$ against
the heat kernel at the origin at time $\omega R^2-t$, which is a Gaussian window
of width $R$ (up to the factor $R^q$). This is the classical backward-semigroup
functional \cite{Weissler1981,Meier1988}.

\smallskip\noindent\textbf{Step 1: an exact identity and an ODE inequality.}
Apply $P_{\omega R^2-t}$ to \eqref{eq:mild} and evaluate at $0$; since
$f\mapsto[P_\tau f](0)$ is continuous on $L^1$, it commutes with the Bochner
integral, and the semigroup law gives
\begin{equation}\label{eq:FR-identity}
 F_R(t)=\eps R^qP_{\omega R^2}u_0(0)+\int_0^t\!\!\int h(s)u(s,x)^pK_R(s,x)\dd x\dd s .
\end{equation}
So $F_R$ is absolutely continuous and nondecreasing, with
$F_R'=h\int u^pK_R\dd x$ a.e. By Jensen's inequality for the probability
measure $R^{-q}K_R\dd x$,
\begin{equation}\label{eq:FR-ode}
 F_R'(t)\ge h(t)R^{-q(p-1)}F_R(t)^p\quad\text{a.e.}
\end{equation}
By \eqref{eq:origin-gaussian}, with $c_{\rm e}=c_{\rm o}\omega^{-q/2}$,
$A_{\rm e}=A_{\rm o}/(\omega-b)$, $C_{\rm e}=C_{\rm o}(\omega-b)^{-q/2}$,
$a_{\rm e}=a_{\rm o}/\omega$, we have for $t\le bR^2$
\begin{equation}\label{eq:KR-envelope}
 c_{\rm e}e^{-A_{\rm e}\CC(0,x)^2/R^2}\le K_R(t,x)\le C_{\rm e}e^{-a_{\rm e}\CC(0,x)^2/R^2}.
\end{equation}

\begin{idea}
The identity says that $F_R$ is the linear part plus the reaction seen through
the window. Inequality \eqref{eq:FR-ode} says that $F_R$ obeys the same ODE as
a spatially constant solution, with the coefficient $h$ rescaled by the size of
the window. The envelope bound \eqref{eq:KR-envelope} says that for
$t\le bR^2$ the window is comparable to Gaussians $e^{-c\CC^2/R^2}$,
uniformly in $t$.
\end{idea}

\smallskip\noindent\textbf{Step 2: the envelope energy and a lower bound.}
To follow the reaction at all scales we measure it through Gaussians of width
$R$ by a single nondecreasing quantity.
Fix $A>0$, decrease $r$ and choose $\ell>1$ so that
$A/r^2\ge A_{\rm e}$, $\ell^2>b$, $A/\ell^2\le a_{\rm e}$, and define the
nondecreasing energy
\[
 E(R):=\int_0^{R^2}\!\!\int h(t)u(t,x)^p\exp\!\Bigl(-A\frac{\CC(0,x)^2}{R^2}\Bigr)\dd x\dd t .
\]
By \eqref{eq:origin-gaussian} and monotone convergence there are
$R_0\ge\max\{1,a^{-1/2}\}$ and $c_{\rm s}>0$, depending only on $X$ and $u_0$,
with $R^qP_{\omega R^2}u_0(0)\ge c_{\rm s}$ for $R\ge R_0$. Restricting the
nonlinear integral in \eqref{eq:FR-identity} to $t\le r^2R^2$ and using the
lower bound in \eqref{eq:KR-envelope}, with $c_F=\min\{c_{\rm s},c_{\rm e}\}$,
\begin{equation}\label{eq:F-lower}
 F_R(aR^2)\ge c_F[\eps+E(rR)],\qquad R\ge R_0,\ bR^2<T .
\end{equation}
In words, \eqref{eq:F-lower} says that at time $aR^2$ the window sees at least
the initial mass plus all the reaction that has taken place at scales up to $rR$.

\smallskip\noindent\textbf{Step 3: an increment inequality and an ODE cap.}
The reaction in the time slab $[aR^2,bR^2]$ is measured by the quantity
$\Lambda(R)$ below, normalized by the size of the window.
Let $\Lambda(R):=R^{-q(p-1)}\int_{aR^2}^{bR^2}h\dd t$. Since $q(p-1)=2\alpha$,
\begin{equation}\label{eq:Lambda-comp}
 b^{-\alpha}\Lambda(R)\le\int_{aR^2}^{bR^2}h\,t^{-\alpha}\dd t\le a^{-\alpha}\Lambda(R).
\end{equation}
Integrating \eqref{eq:FR-ode} over $[aR^2,bR^2]$ and using \eqref{eq:F-lower},
$F_R(bR^2)-F_R(aR^2)\ge c_F^p\Lambda(R)[\eps+E(rR)]^p$. On $[aR^2,bR^2]$ the
upper bound in \eqref{eq:KR-envelope} is at most
$C_{\rm e}e^{-A\CC^2/(\ell R)^2}$, and $[aR^2,bR^2]\subset[r^2R^2,\ell^2R^2]$.
If moreover $\ell^2R^2<T$, so that $E(\ell R)$ is defined, the left side is
therefore at most $C_{\rm e}[E(\ell R)-E(rR)]$. Hence
\begin{equation}\label{eq:E-increment}
 E(\ell R)-E(rR)\ge c_{\rm inc}\Lambda(R)[\eps+E(rR)]^p,\qquad c_{\rm inc}=c_F^p/C_{\rm e},
 \qquad R\ge R_0,\ \ell^2R^2<T.
\end{equation}
Moreover, since $F_R>0$ is absolutely continuous, \eqref{eq:FR-ode} gives
$(F_R^{1-p})'\le-(p-1)hR^{-q(p-1)}$, and integrating over the slab,
$0<F_R(bR^2)^{1-p}\le F_R(aR^2)^{1-p}-(p-1)\Lambda(R)$, so by \eqref{eq:F-lower}
(the ODE cap)
\begin{equation}\label{eq:cap}
 \Lambda(R)[\eps+E(rR)]^{p-1}\le C_{\rm cap}:=[(p-1)c_F^{p-1}]^{-1}.
\end{equation}
Unlike \eqref{eq:E-increment}, the cap \eqref{eq:cap} uses only $R\ge R_0$ and
$bR^2<T$.

\begin{idea}
Inequality \eqref{eq:E-increment} says that the energy grows from scale $rR$ to
scale $\ell R$ by at least $c\,\Lambda(R)[\eps+E(rR)]^p$: the more reaction in
the slab, the more energy. Inequality \eqref{eq:cap} says that
$\Lambda(R)[\eps+E(rR)]^{p-1}$ cannot be large, because otherwise the ODE
inequality \eqref{eq:FR-ode} would force $F_R$ to blow up inside the slab.
\end{idea}

\smallskip\noindent\textbf{Step 4: telescoping on a geometric grid.}
Now we chain these inequalities along a geometric sequence of scales.
Set $\mu=\ell/r$, $S_j=S_0\mu^j$ with $S_0\ge rR_0$, $Y_j=\eps+E(S_j)$ and
$W_j=\Lambda(S_j/r)$, assuming $(\mu S_{N-1})^2<T$. Then \eqref{eq:E-increment}
and \eqref{eq:cap} read $Y_{j+1}\ge Y_j(1+z_j)$ with
$z_j=c_{\rm inc}W_jY_j^{p-1}\in[0,Z]$, $Z=c_{\rm inc}C_{\rm cap}$. With
$c_Z:=\min_{0\le z\le Z}(1-(1+z)^{1-p})/z>0$,
\[
 Y_j^{1-p}-Y_{j+1}^{1-p}\ge Y_j^{1-p}[1-(1+z_j)^{1-p}]\ge c_Zc_{\rm inc}W_j,
\]
and summing over $j$,
\begin{equation}\label{eq:sum-W}
 \sum_{j=0}^{N-1}W_j\le\frac{\eps^{-(p-1)}}{c_Zc_{\rm inc}},
\end{equation}
since $Y_0\ge\eps$. This is the step that sums the contributions of all
parabolic scales; a single slab only gives $\Lambda(R)\le C\eps^{-(p-1)}$,
which is useless at criticality.

\smallskip\noindent\textbf{Step 5: interlaced grids for arbitrary $h$.}
Inequality \eqref{eq:sum-W} controls the reaction only in the slabs attached to
one grid of scales, and these slabs leave gaps in time. We cover the gaps with
finitely many shifted grids, and handle the end of the interval separately.

\smallskip\noindent\emph{(5a) Interlaced grids.}
Write $W(S):=\Lambda(S/r)$, so that $W_j=W(S_j)$. The slab attached to $W(S)$ is
$I(S)=[A_*S^2,B_*S^2]$ with $A_*=a/r^2$ and $B_*=b/r^2$. A single grid of ratio $\mu$ leaves gaps between consecutive slabs.
Let $K_{\rm g}\ge1$ be the least integer with $\eta:=\mu^{1/K_{\rm g}}$
satisfying $\eta^2\le b/a$, and put $\widehat S_i=S_*\eta^i$ with $S_*=rR_0$.
Consecutive slabs $I(\widehat S_i)$ then overlap, so they cover
$[A_*S_*^2,B_*\widehat S_{i_0}^2]$, and each residue class
$i\equiv k\pmod{K_{\rm g}}$ is a geometric grid of ratio $\mu$. Put
$T_1=bR_0^2$ and $\kappa=b/(\ell^2\eta^2)\in(0,1)$.

\smallskip\noindent\emph{(5b) The bulk of the interval.}
Let $\mu^2S_*^2\le\tau<T$. Applying \eqref{eq:sum-W} to each of the $K_{\rm g}$
residue classes and using \eqref{eq:Lambda-comp} we get
\begin{equation}\label{eq:interior}
 \int_{aR_0^2}^{\kappa\tau}h\,t^{-\alpha}\dd t\le\frac{a^{-\alpha}K_{\rm g}}{c_Zc_{\rm inc}}\eps^{-(p-1)} .
\end{equation}
Indeed, the admissible slabs cover up to $B_*\widehat S_{i_\tau}^2>\kappa\tau$.

\smallskip\noindent\emph{(5c) The end of the interval.}
The remaining piece $[\kappa\tau,\tau]$ is handled by the cap alone. For
$\tau\ge T_1$ and $J\ge1$ with $(a/b)^{J-1}\tau\ge T_1$, put
$\tau_j=(a/b)^j\tau$ and $R_j^2=\tau_j/b$. Then
$[aR_j^2,bR_j^2]=[\tau_{j+1},\tau_j]$, and \eqref{eq:cap} with $E\ge0$ gives
\begin{equation}\label{eq:terminal}
 \int_{(a/b)^J\tau}^\tau h\,t^{-\alpha}\dd t\le Ja^{-\alpha}C_{\rm cap}\eps^{-(p-1)} .
\end{equation}

\smallskip\noindent\emph{(5d) Putting the pieces together.}
Let $J_0\ge1$ be the least integer with $(a/b)^{J_0}\le\kappa$, let
$\tau_1=T_1/\kappa$, and let $J_1=1+\log(1/\kappa)/\log(b/a)$. For
$T_1\le\tau\le\tau_1$, \eqref{eq:terminal} with the largest admissible $J$,
that is, the one with $(a/b)^J\tau<T_1\le(a/b)^{J-1}\tau$, covers $[T_1,\tau]$,
and this $J$ satisfies $J<J_1$. For $\tau>\tau_1$, \eqref{eq:interior} covers
$[T_1,\kappa\tau]$ and \eqref{eq:terminal} with $J=J_0$ covers $[\kappa\tau,\tau]$.
This gives the first inequality in \eqref{eq:upper-large} with
$C_{\rm up}=a^{-\alpha}\max\{J_1C_{\rm cap},K_{\rm g}/(c_Zc_{\rm inc})+J_0C_{\rm cap}\}$.
The second follows from $(1+t)^{-\alpha}\le t^{-\alpha}$. $\qed$

The interlacing is the only place where arbitrary $h$ requires extra work, and
it is geometric rather than analytic.

\subsection{Proof of Theorem~\ref{thm:master}-(i),(ii) and Theorem~\ref{thm:threshold}}

Parts (a) and (b) of Theorem~\ref{thm:master} already contain the lifespan law;
what remains is to combine them. Roughly, (a) says that blow-up cannot occur
while $\eps^{p-1}\Hp(T)\le c_{\rm fp}$, and (b) says that it must occur before
$\Hp$ exceeds $\Hp(T_1)+C_{\rm up}\eps^{-(p-1)}$.

\begin{proof}[Proof of Theorem~\ref{thm:master}-(i),(ii)]
(i) is (a) with $T=\infty$. For (ii), if $\Hp(\infty)=\infty$ and $T_\eps=\infty$,
letting $\tau\to\infty$ in \eqref{eq:upper-large} gives a contradiction, so
$T_\eps<\infty$. Suppose $\eps^{p-1}\Hp(T_\eps)<c_{\rm fp}$. Continuity of
$\Hp$ gives $T>T_\eps$ with $\eps^{p-1}\Hp(T)\le c_{\rm fp}$, contradicting (a);
so $\Hp(T_\eps)\ge c_{\rm fp}\eps^{-(p-1)}$. For
$\eps\le\eps_0$ we have $\Hp(T_1)\le C_{\rm up}\eps^{-(p-1)}$, and
\eqref{eq:upper-large} with $\tau\uparrow T_\eps$ gives the upper bound.
\end{proof}

\begin{proof}[Proof of Theorem~\ref{thm:threshold}]
(i) is Lemma~\ref{lem:comparison} and the definition of $\eps_*$.

(ii) If $\Hp(\infty)=\infty$, Theorem~\ref{thm:master}-(ii) gives $T_\eps<\infty$
for every $\eps>0$, so $\{\eps:T_\eps=\infty\}=\emptyset$ and $\eps_*=0$. If
$\Hp(\infty)<\infty$, Theorem~\ref{thm:master}-(i) gives $T_\eps=\infty$ for
$\eps^{p-1}\Hp(\infty)\le c_{\rm fp}$, which is the lower bound. For the upper
bound let $T_\eps=\infty$ and let $\tau\to\infty$ in \eqref{eq:upper-large}:
since $(1+t)^{-\alpha}\le t^{-\alpha}$,
\[
 \Hp(\infty)-\Hp(T_1)=\int_{T_1}^\infty h(t)(1+t)^{-\alpha}\dd t\le C_{\rm up}\eps^{-(p-1)} .
\]
If $\Hp(\infty)\ge2\Hp(T_1)$ the left side is at least $\frac12\Hp(\infty)$, so
$\eps^{p-1}\le2C_{\rm up}/\Hp(\infty)$ for every $\eps$ with $T_\eps=\infty$,
hence for their supremum $\eps_*$.

(iii) For large data one window and one time interval on which $h$ has positive
mass suffice: the inequality \eqref{eq:FR-ode} then forces blow-up once $\eps$
is large. If $h=0$ a.e., $u=\eps P_tu_0$ is global. Otherwise take $t_0<t_1$ with
$M=\int_{t_0}^{t_1}h>0$ and $R$ with $\omega R^2>t_1$, and suppose $T_\eps>t_1$.
The identity \eqref{eq:FR-identity} and \eqref{eq:FR-ode} hold on $[0,t_1]$ (the
kernel $\Gamma_{\omega R^2-t}(0,\cdot)$ is bounded there), and $F_R(t)\ge
F_R(0)=\eps R^qm_R$ with $m_R=P_{\omega R^2}u_0(0)>0$. Then
$(F_R^{1-p})'\le-(p-1)R^{-q(p-1)}h$, and integrating over $[t_0,t_1]$,
\[
 (p-1)R^{-q(p-1)}M<F_R(t_0)^{1-p}\le\eps^{1-p}R^{-q(p-1)}m_R^{1-p},
\]
i.e.\ $\eps<\eps_1:=m_R^{-1}((p-1)M)^{-1/(p-1)}$. So $\eps\ge\eps_1$ forces
$T_\eps\le t_1$, and $\eps_*\le\eps_1<\infty$.
\end{proof}

\begin{remark}
The origin is a convenient center because \eqref{eq:ball-origin} holds for all
radii, but it is not essential: for fixed $x_0$ one has
$|\BX(x_0,R)|\asymp R^q$ for $R\ge2\CC(0,x_0)$, and the argument can be run at
large scales around $x_0$.
\end{remark}

\subsection{Proof of Theorem~\ref{thm:apriori} and Corollary~\ref{cor:small}}\label{subsec:apriori-proof}

These are energy-type estimates: they use only the maximum principle, the
$L^k$-contractivity of the semigroup and a comparison with an ordinary
differential equation.

By Lemma~\ref{lem:classical-mild}, $u$ is a mild solution on $[0,T)$ and $T\le T_\eps$; only \eqref{eq:mild} is used below.

\emph{(a)} The Duhamel term in \eqref{eq:mild} is nonnegative, so
$u(t)\ge\eps P_tu_0$. Integrating \eqref{eq:mild} in $x$ and using that $P_\tau$
conserves mass gives the identity, since $\norm{u(s)}_p^p=\norm{u(s)^p}_1$.

\emph{(b)} $P_\tau$ is a contraction on $L^k$ and
$\norm{u^p}_k\le\norm u_\infty^{p-1}\norm u_k$, so \eqref{eq:mild} gives
\[
 \norm{u(t)}_k\le\eps\norm{u_0}_k+\int_0^th(s)\norm{u(s)}_\infty^{p-1}\norm{u(s)}_k\dd s .
\]
The functions involved are measurable and bounded on $[0,t]$ (the sup norm is
lower semicontinuous on $L^1$), so Gronwall's lemma gives (b).

\emph{(c)} For the last claim apply the argument to the maximal mild solution, i.e.\ with
$T=T_\eps$. Let $t_0<T$ with $\Theta_\eps(t_0)<1$, $v_0=\eps\norm{u_0}_\infty$ and
$v(t)=v_0(1-\Theta_\eps(t))^{-1/(p-1)}$ for $t\le t_0$. This is the absolutely
continuous solution of $v=v_0+\int_0^thv^p$, and since $P_\tau1=1$ and
$P_tu_0\le\norm{u_0}_\infty$,
\[
 v(t)\ge\eps P_tu_0+\int_0^th(s)P_{t-s}\bigl(v(s)^p\bigr)\dd s .
\]
Let $w_0=0$ and $w_{k+1}=\eps P_tu_0+\int_0^thP_{t-s}(w_k^p)\dd s$. By positivity
of $P_\tau$ and induction, $0\le w_k\le w_{k+1}\le\min\{u,v\}$. As in the proof
of Lemma~\ref{lem:comparison}, $w=\sup_kw_k$ is a bounded mild solution on
$[0,t_0]$, so $w=u$ by Lemma~\ref{lem:local}-(i), and therefore $u\le v$. Because
$hv^{p-1}=v'/v$,
\[
 \int_0^th\norm u_\infty^{p-1}\dd s\le\int_0^t\frac{v'}v\dd s
 =\log\frac{v(t)}{v_0}=-\frac1{p-1}\log\bigl(1-\Theta_\eps(t)\bigr),
\]
and (b) gives \eqref{eq:Lk-ODE}. If $T_\eps<T_*$, then $\norm{u(t)}_\infty\le v(t)$
stays bounded on $[0,T_\eps)$, contradicting Lemma~\ref{lem:local}-(iii); hence
$T_\eps\ge T_*$. \qed

\begin{proof}[Proof of Corollary~\ref{cor:small}]
By Lemma~\ref{lem:classical-mild}, $u$ is a mild solution, and by
Theorem~\ref{thm:master}-(a), $N:=N_T(u)\le C\eps$ with $C=2C_{\rm L}\mathcal U_0$.
Interpolating between $\norm{u(t)}_1\le N$ and
$\norm{u(t)}_\infty\le N(1+t)^{-q/2}$ gives the first bound. Next,
$u-\eps P_tu_0=\mathcal N(u)$ and by Subsection~\ref{subsec:lower},
$\norm{\mathcal N(u)(t)}_1\le\Hp(T)N^p$ and
$(1+t)^{q/2}\norm{\mathcal N(u)(t)}_\infty\le C_{\rm N}\Hp(T)N^p$; interpolate
again. The mass bounds follow from Theorem~\ref{thm:apriori}-(a) and
$\norm{u(s)^p}_1\le N^p(1+s)^{-\alpha}$.
\end{proof}

\section{Proofs of the explicit laws}\label{sec:applications}

All explicit laws are consequences of the master law \eqref{eq:master}; the work
is to compute $\Hp$ for the given $h$ and to invert. Let $0<\eps\le\eps_0$. If
$\Hp(\infty)=\infty$, Theorem~\ref{thm:master}-(ii) gives
$\Hp(T_\eps)\ge c_{\rm fp}\eps^{-(p-1)}\to\infty$ as $\eps\to0$; since $\Hp$ is
bounded on bounded intervals, $T_\eps\to\infty$. Therefore: if
$\Hp(T)\asymp G(T)$ for large $T$, then
\begin{equation}\label{eq:G-transfer}
 G(T_\eps)\asymp\eps^{-(p-1)}\qquad(\eps\to0).
\end{equation}
If $\Hp(\infty)<\infty$, small data are global by
Theorem~\ref{thm:master}-(i). The task is thus to compute $\Hp$ and invert
\eqref{eq:G-transfer}; the elementary asymptotics needed for this are in
Subsection~\ref{subsec:asympt}.

\begin{proof}[Proof of Theorem~\ref{thm:powerlog}]
(a) \emph{The case $h\equiv1$.} Here $\Hp(T)=\int_0^T(1+t)^{-\alpha}\dd t$. If
$\alpha<1$ (that is, $p<p_F$), then $\Hp(T)\asymp T^{1-\alpha}$, and
\eqref{eq:G-transfer} gives $T_\eps^{1-\alpha}\asymp\eps^{-(p-1)}$. If
$\alpha=1$ (that is, $p=1+2/q$), then $\Hp(T)=\log(1+T)$, and
\eqref{eq:G-transfer} gives $\log T_\eps\asymp\eps^{-(p-1)}$. If $\alpha>1$, then
$\Hp(\infty)<\infty$ and small data are global.

(b) \emph{The case $h\asymp(1+t)^\beta$.} Since
$h(t)(1+t)^{-\alpha}\asymp t^{\rho-1}$ for large $t$, we have
$\Hp(T)\asymp1+\int_{T_h}^Tt^{\rho-1}\dd t$, which is $\asymp T^\rho$,
$\asymp\log T$, or bounded, according as $\rho>0$, $\rho=0$, or $\rho<0$. Now
\eqref{eq:G-transfer} gives the three laws. The condition $\rho=0$ means
$\frac q2(p-1)=1+\beta$, i.e.\ $p=1+2(1+\beta)/q$; and if $\beta\le-1$, then
$\rho\le-\alpha<0$ for every $p>1$.

(c) \emph{Logarithmic factors.} Here $h(t)t^{-\alpha}\asymp t^{\rho-1}(\log t)^\gamma$.
By Lemma~\ref{lem:asympt}-(i), $\Hp(T)\asymp T^\rho(\log T)^\gamma$ for
$\rho>0$, and $\Hp(\infty)<\infty$ for $\rho<0$. For $\rho>0$ we invert with
Lemma~\ref{lem:asympt}-(ii), taking $\Sigma=\eps^{-(p-1)}$; since
$\log\Sigma\asymp\log\frac1\eps$, this gives the stated law.

(d) \emph{Critical hierarchy.} Now $\alpha=1+\beta$, so
$(1+t)^\beta t^{-\alpha}\asymp t^{-1}$ and
$h(t)t^{-\alpha}\asymp t^{-1}\prod_{j<j_*}L_j^{-1}\,L_{j_*}^\chi$. The iterated
logarithms satisfy $L_1'=(e+t)^{-1}$ and $L_{j+1}'=L_j'/(e+L_j)$, so by
induction $L_{j_*}'\asymp t^{-1}\prod_{j<j_*}L_j^{-1}$. Therefore
$h\,t^{-\alpha}\asymp L_{j_*}^\chi L_{j_*}'$, and integrating,
\[
 \Hp(T)\asymp L_{j_*}(T)^{\chi+1}\ (\chi>-1),\qquad
 \Hp(T)\asymp\log L_{j_*}(T)\asymp L_{j_*+1}(T)\ (\chi=-1),
\]
where for $\chi=-1$ we used $0\le L_{j_*+1}-\log L_{j_*}\le\log(1+e)$, while
$\Hp(\infty)<\infty$ for $\chi<-1$. Applying \eqref{eq:G-transfer} gives
$L_{j_*}(T_\eps)\asymp\eps^{-(p-1)/(\chi+1)}$ and
$L_{j_*+1}(T_\eps)\asymp\eps^{-(p-1)}$, respectively.
\end{proof}

\begin{proof}[Proof of Theorem~\ref{thm:fast}]
(a) \emph{Exponential growth.} We have
$h(t)(1+t)^{-\alpha}\asymp e^{\theta t}t^{\beta-\alpha}$, so by
Lemma~\ref{lem:asympt}-(iii) $\Hp(T)\asymp e^{\theta T}T^{\beta-\alpha}$. Taking
logarithms in \eqref{eq:G-transfer},
\[
 \theta T_\eps+(\beta-\alpha)\log T_\eps=(p-1)\log\tfrac1\eps+O(1).
\]
Thus $T_\eps\sim\frac{p-1}\theta\log\frac1\eps$, hence
$\log T_\eps=\log\log\frac1\eps+O(1)$, and substituting this back into the
display gives the claim.

(b) \emph{Integrable reactions.} In both cases $\Hp(\infty)<\infty$ for every
$p>1$, so small data are global by Theorem~\ref{thm:master}-(i). If
$\supp h\subset[0,T_{\rm off}]$, then \eqref{eq:mild} and the semigroup law give
$u(t)=P_{t-T_{\rm off}}u(T_{\rm off})$ for $t\ge T_{\rm off}$: after $T_{\rm off}$ the solution just follows
the heat flow.

(c) \emph{Periodic reactions.} Let $w(t)=(1+t)^{-\alpha}$, which is decreasing,
and let $\bar h$ be the mean of $h$ over a period. Then, for every $k$,
\[
 \varpi\bar h\,w((k+1)\varpi)\le\int_{k\varpi}^{(k+1)\varpi}hw\le\varpi\bar h\,w(k\varpi),
\]
and $w(k\varpi)/w((k+1)\varpi)\le(1+\varpi)^\alpha$. Summing over the periods
contained in $[0,T]$ and treating the incomplete last period by the
monotonicity of $w$, we get $\Hp(T)\asymp\int_1^T(1+t)^{-\alpha}\dd t$ for large
$T$ if this integral diverges (i.e.\ $\alpha\le1$), and $\Hp(\infty)<\infty$
otherwise. These are the same relations as for $h\equiv1$, so the laws of
Theorem~\ref{thm:powerlog}-(a) follow.

(d) \emph{Dense blocks.} On $[4^k,4^{k+1}]$ the inactive part
$(2\cdot4^k,4^{k+1}]$ has twice the length of the active part
$[4^k,2\cdot4^k]$, and the weight $t^{-\alpha}$ is smaller there. So the
inactive part carries at most twice the weighted mass of the active part, i.e.\
each active interval carries at least one third of the weighted mass of
$[4^k,4^{k+1}]$. Since also $(1+t)^{-\alpha}\asymp t^{-\alpha}$ for
$t\ge1$, we get $\int_{\mathcal A\cap[1,T]}t^{-\alpha}\dd t\asymp\int_1^Tt^{-\alpha}\dd t$
for large $T$, and the claim follows as in (c). (The comparison fails for $T$
near $4$, where the left side vanishes.)

(e) \emph{Sparse blocks.} The block $[2^k,2^k+2^{\vartheta k}]$ has length
$2^{\vartheta k}$ and lies in $[2^k,3\cdot2^k]$, so it carries weighted mass
between $3^{-\alpha}2^{(\vartheta-\alpha)k}$ and $2^{(\vartheta-\alpha)k}$. For
$2^i\le T<2^{i+1}$ the integral over $\mathcal A_\vartheta\cap[1,T]$ therefore
lies between $3^{-\alpha}\sum_{k<i}2^{(\vartheta-\alpha)k}$ and
$\sum_{k\le i}2^{(\vartheta-\alpha)k}$. Hence $\Hp(T)\asymp T^{\vartheta-\alpha}$,
$\Hp(T)\asymp\log T$, or $\Hp(\infty)<\infty$ according as $\alpha<\vartheta$,
$\alpha=\vartheta$, or $\alpha>\vartheta$, and \eqref{eq:G-transfer} gives the
laws.

\emph{The last statement.} Suppose $h=0$ a.e.\ on $(t_1,t_2)$ and
$T_\eps\in(t_1,t_2]$. For $t_1\le t<T_\eps$, \eqref{eq:mild} and the semigroup
law give $u(t)=P_{t-t_1}u(t_1)$, so $\norm{u(t)}_\infty\le\norm{u(t_1)}_\infty$
stays bounded as $t\uparrow T_\eps$. This contradicts
Lemma~\ref{lem:local}-(iii).
\end{proof}

\begin{proof}[Proof of Theorem~\ref{thm:nonstandard}]
(i) \emph{Delayed pulse.} Since $\supp h\subset[T_0,T_0+1]$, we have
$\Hp(\infty)\le M(1+T_0)^{-\alpha}$. By Theorem~\ref{thm:threshold}-(ii) and
$\alpha/(p-1)=q/2$,
\[
 \eps_*\ge c_{\rm fp}^{1/(p-1)}(1+T_0)^{q/2}M^{-1/(p-1)}.
\]
For the upper bound we use the proof of Theorem~\ref{thm:threshold}-(iii) with
$t_0=T_0$, $t_1=T_0+1$ and $R=\max\{R_0,\sqrt{2(T_0+1)/\omega}\}\le C(1+T_0)^{1/2}$.
It gives $\eps_*\le m_R^{-1}((p-1)M)^{-1/(p-1)}$, and
$m_R\ge c_{\rm s}R^{-q}$ by Step~2 of Subsection~\ref{subsec:upper}, so
$\eps_*\le CR^qM^{-1/(p-1)}\le C(1+T_0)^{q/2}M^{-1/(p-1)}$.

(ii) \emph{Sparse bursts.} For $\vartheta=0$ and any $p>1$ we have
$\alpha>\vartheta$, so small data are global by Theorem~\ref{thm:fast}-(e),
while $h\not\to0$ and $\int h=\infty$. Now let $0<\vartheta<1$ and
$\alpha\le\vartheta$. Then $\Hp(\infty)=\infty$, so $T_\eps<\infty$. The
coefficient $h$ vanishes on each gap $(2^k+2^{\vartheta k},2^{k+1})$ and on
$(0,2)$, so the last statement of Theorem~\ref{thm:fast} excludes $T_\eps$ from
$(2^k+2^{\vartheta k},2^{k+1}]$ and from $(0,2]$. Every point of $(2,\infty)$
lies in some $(2^k,2^{k+1}]$, so the only possibility left is
$T_\eps\in(2^k,2^k+2^{\vartheta k}]$.

(iii) \emph{Staircase lifespan.} We first explain the idea. The blocks
$[s_{2j},s_{2j+1})$ are separated by enormously long gaps, because
$s_{k+1}=s_k^2$. Before time $s_{2j}$ the clock $\Hp$ has accumulated only about
$s_{2j-1}^{1-\alpha}=s_{2j}^{(1-\alpha)/2}$, and during $[s_{2j},2s_{2j}]$ it
increases by about $s_{2j}^{1-\alpha}$, which is much more. So if the level
$\eps^{-(p-1)}$ lies between these two values, the lifespan is trapped in
$(s_{2j},2s_{2j}]$ and does not move as $\eps$ varies over a long range. Here
are the details.

\smallskip\noindent\emph{Setting.} Write $\mathcal W(a,b)=\int_a^b(1+t)^{-\alpha}\dd t$
and $\Sigma=\eps^{-(p-1)}$, so that
$\Hp(T)=\int_{\mathcal A\cap[0,T]}(1+t)^{-\alpha}\dd t\le T^{1-\alpha}/(1-\alpha)$.
By Theorem~\ref{thm:master}, $c_{\rm fp}\Sigma\le\Hp(T_\eps)\le
T_\eps^{1-\alpha}/(1-\alpha)$, so the $\liminf$ in the statement is at least
$(p-1)/(1-\alpha)$.

\smallskip\noindent\emph{Two bounds on the clock.} Since $s_{k+1}=s_k^2$, the
terms $\mathcal W(s_{2i},s_{2i+1})\le s_{2i+1}^{1-\alpha}/(1-\alpha)$ grow at least
geometrically, so for some $C_P=C_P(\alpha)$
\[
 \Hp(s_{2j})\le C_Ps_{2j-1}^{1-\alpha}=C_Ps_{2j}^{(1-\alpha)/2},
 \qquad
 \Hp(2s_{2j})-\Hp(s_{2j})\ge c_1s_{2j}^{1-\alpha},\ c_1=3^{-\alpha}.
\]

\smallskip\noindent\emph{Trapping the lifespan.} Put
$\Sigma_j^-=\frac{2C_P}{c_{\rm fp}}s_{2j}^{(1-\alpha)/2}$ and
$\Sigma_j^+=\frac{c_1}{4C_{\rm up}}s_{2j}^{1-\alpha}$; the interval
$[\Sigma_j^-,\Sigma_j^+]$ is nonempty for large $j$. Let
$\Sigma\in[\Sigma_j^-,\Sigma_j^+]$. Then
$\Hp(T_\eps)\ge c_{\rm fp}\Sigma>\Hp(s_{2j})$, so $T_\eps>s_{2j}$. Suppose that
$T_\eps>2s_{2j}$. For large $j$ we have $\eps\le\eps_0$, so
\eqref{eq:upper-large} at $\tau=2s_{2j}$ gives
$c_1s_{2j}^{1-\alpha}\le\Hp(2s_{2j})\le2C_{\rm up}\Sigma\le\frac12c_1s_{2j}^{1-\alpha}$,
a contradiction. Hence $s_{2j}<T_\eps\le2s_{2j}$ for all $\eps\in[\eps_j^-,\eps_j^+]$,
where $\eps_j^\pm=(\Sigma_j^\mp)^{-1/(p-1)}$, and
$\eps_j^+/\eps_j^-\ge s_{2j}^{c_*}$ with $c_*=(1-\alpha)/(4(p-1))$ for large $j$.

\smallskip\noindent\emph{No power law.} At the amplitude corresponding to
$\Sigma_j^+$ we have $\log T_\eps/\log(1/\eps)\to(p-1)/(1-\alpha)$, and at the
amplitude corresponding to $\Sigma_j^-$ it tends to $2(p-1)/(1-\alpha)$. The two
limits differ, which gives the statements on $\liminf$ and $\limsup$, and shows
that no power law $T_\eps\asymp\eps^{-\kappa}$ exists.
\end{proof}

\section{Beyond power nonlinearities}\label{sec:general}

The power $u^p$ enters the analysis in two ways: through monotonicity and
convexity, which are soft properties, and through exact scale invariance,
$f(\lambda s)=\lambda^pf(s)$, which drives the upper estimate. In this section
we record what survives for
\begin{equation}\label{eq:main-f}
 u_t-\Lcal u=h(t)f(u),\qquad u(0)=\eps u_0,
\end{equation}
under the following assumption.

\smallskip
\noindent\textbf{(F)} \ $f:[0,\infty)\to[0,\infty)$ is locally Lipschitz,
$f(0)=0$, $f>0$ on $(0,\infty)$, and $s\mapsto f(s)/s$ is nondecreasing.
\smallskip

\noindent We say that $f$ satisfies \emph{Osgood's condition} if
$\int_1^\infty ds/f(s)<\infty$, that is, if the ordinary differential equation
$y'=f(y)$ blows up in finite time. Mild solutions and the lifespan $T^f_\eps$ are defined as before.
Lemma~\ref{lem:local} holds with $pB^{p-1}$ replaced by the Lipschitz constant
of $f$ on $[0,B]$ (only local Lipschitz continuity and $f(0)=0$ are needed), and
the comparison Lemma~\ref{lem:comparison} holds because $f$ is nondecreasing
under (F).

\begin{inwords}
Osgood's condition on $f$ is necessary for blow-up
(Proposition~\ref{prop:osgood}); if $f$ is squeezed between multiples of $s^p$,
the whole lifespan law persists (Theorem~\ref{thm:sandwich}); for a general $f$
we can prove only a lower bound for the lifespan
(Theorem~\ref{thm:Hf-lower}), and an upper bound that is sharp only when the
weighted integral is dominated by late times (Proposition~\ref{prop:window}).
The sharp upper bound for a general $f$ is open.
\end{inwords}

\subsection{Osgood's condition and nonlinearities comparable to a power}

\begin{proposition}\label{prop:osgood}
If $f$ satisfies \textup{(F)} and $\int_1^\infty ds/f(s)=\infty$, then
$T^f_\eps=\infty$ for every $\eps>0$, $u_0$ and $h$ as in \eqref{eq:h-ass}.
\end{proposition}

\begin{proof}
Suppose $T^f_\eps<\infty$ and put $m(t)=\norm{u(t)}_\infty$ for
$t<T^f_\eps$. Since $P_{t-s}$ is a contraction on $L^\infty$,
\eqref{eq:mild} (with $u^p$ replaced by $f(u)$) gives
\[
 m(t)\le c+\int_0^th(s)\,f\bigl(m(s)\bigr)\dd s,\qquad c:=\eps\norm{u_0}_\infty>0 .
\]
Let $y(t)$ be the right-hand side. Then $y$ is absolutely continuous and,
since $f$ is nondecreasing under \textup{(F)} and $m\le y$,
$y'=h\,f(m)\le h\,f(y)$ a.e. Put $\mathcal G(x)=\int_c^x d\sigma/f(\sigma)$ for
$x\ge c$. Dividing by $f(y)>0$ gives $(\mathcal G(y))'\le h$, so
\[
 \mathcal G\bigl(y(t)\bigr)\le\int_0^th\dd s\le\int_0^{T^f_\eps}h\dd s<\infty
 \qquad(t<T^f_\eps).
\]
This is the Bihari--LaSalle inequality. Since $\int_1^\infty ds/f(s)=\infty$,
we have $\mathcal G(x)\to\infty$ as $x\to\infty$, so $y$, and hence $m\le y$, is
bounded on $[0,T^f_\eps)$. This contradicts Lemma~\ref{lem:local}-(iii). Hence
$T^f_\eps=\infty$.
\end{proof}

So Osgood's condition is necessary for blow-up; the two-sided law cannot hold
for, e.g., $f(s)=s\log(e+s)$. Comparison with powers transfers the whole law
to nonlinearities comparable to $s^p$.

\begin{inwords}
If $f$ is squeezed between two multiples of $s^p$, where the lower bound holds
for all $s$ and the upper bound near $s=0$, then the lifespan law of
Theorem~\ref{thm:master} persists, with different constants. This covers
$s^p+s^r$ with $r\ge p$, $s^p\log(e+s)$, $s^pe^s$ and $s^p(1+s)^m$ with $m\ge0$.
\end{inwords}

\begin{theorem}\label{thm:sandwich}
Let $f$ be nondecreasing and locally Lipschitz with $f(0)=0$, and for some
$p>1$, $0<c_0\le C_f$, $\delta>0$,
\begin{equation}\label{eq:sandwich}
 f(s)\ge c_0s^p\ \ (s\ge0),\qquad f(s)\le C_f s^p\ \ (0\le s\le\delta).
\end{equation}
Let $c_{\rm fp},C_{\rm up},T_1,C$ be as in Theorem~\ref{thm:master} and
\eqref{eq:X-bound}, and $\eps_\delta=\delta/C$.
\begin{enumerate}[label=\textup{(\roman*)}]
\item If $\eps\le\eps_\delta$ and $\eps^{p-1}C_f\Hp(T)\le c_{\rm fp}$, then
$T^f_\eps\ge T$. Hence, if $T^f_\eps<\infty$ and $\eps\le\eps_\delta$,
$\Hp(T^f_\eps)\ge(c_{\rm fp}/C_f)\eps^{-(p-1)}$.
\item If $\Hp(\infty)=\infty$, then $T^f_\eps<\infty$ for every $\eps>0$ and
$\Hp(T^f_\eps)\le\Hp(T_1)+(C_{\rm up}/c_0)\eps^{-(p-1)}$.
\item If $\Hp(\infty)<\infty$, then $T^f_\eps=\infty$ for
$\eps\le\min\{\eps_\delta,(c_{\rm fp}/(C_f\Hp(\infty)))^{1/(p-1)}\}$.
\end{enumerate}
In particular $\Hp(T^f_\eps)\asymp\eps^{-(p-1)}$ for small $\eps$ when
$\Hp(\infty)=\infty$, and Theorems~\ref{thm:powerlog}--\ref{thm:fast} apply
verbatim to $f(s)=s^p+s^r$ $(r\ge p)$, $s^p\log(e+s)$, $s^pe^s$, $s^p(1+s)^m$
$(m\ge0)$.
\end{theorem}

\begin{proof}
(i) Apply Theorem~\ref{thm:master}-(a) to the coefficient $C_f h$: there is a
mild solution $v$ of $v_t-\Lcal v=C_f hv^p$ on $[0,T)$ with
$\norm v_\infty\le C\eps\le\delta$. The monotone iterates $w_k$ of
\eqref{eq:main-f} satisfy $w_k\le v$, because $f(v)\le C_f v^p$ for
$v\le\delta$ and $f$ is nondecreasing. Their limit is a mild solution of
\eqref{eq:main-f} on $[0,T)$ (as in Lemma~\ref{lem:comparison}).

(ii) Let $v$ solve $v_t-\Lcal v=c_0hv^p$, $v(0)=\eps u_0$, with lifespan
$T^\sharp<\infty$ (Theorem~\ref{thm:master}-(ii) for $c_0h$). Since
$f\ge c_0s^p$, the iterates satisfy $v_k\le w_k\le u^f$ as long as $u^f$ exists,
so $T^f_\eps\le T^\sharp$. Applying \eqref{eq:upper-large} to $v$, with the
coefficient $c_0h$ and $\tau<T^f_\eps$, and then letting $\tau\uparrow T^f_\eps$,
gives the bound in (ii).

(iii) is (i) with $T=\infty$.
\end{proof}

For continuous $h>0$, qualitative blow-up and global-existence criteria for
sums of powers and for exponential nonlinearities are given in
\cite[Sect.~7]{MeglioliPunzo2026}; Theorem~\ref{thm:sandwich} adds the
two-sided lifespan law, for an arbitrary $h\ge0$ in $L^1_{\rm loc}$.

\subsection{A lower bound for general \texorpdfstring{$f$}{f}}

The heuristic of Subsection~\ref{subsec:heuristic} suggests the right quantity.
At time $s$ the linear flow has size $\eps(1+s)^{-q/2}$ on a set of volume
$(1+s)^{q/2}$, so the Duhamel integrand has mass about
$(1+s)^{q/2}f(\eps(1+s)^{-q/2})$; for $f=s^p$ this is $\eps^p(1+s)^{-\alpha}$.

\begin{theorem}\label{thm:Hf-lower}
Assume (F). There are $A,\kappa>0$ depending only on $X$ and $\mathcal U_0$ such
that if
\begin{equation}\label{eq:Hf-cond}
 \int_0^Th(s)(1+s)^{q/2}f\bigl(A\eps(1+s)^{-q/2}\bigr)\dd s\le\kappa\eps,
\end{equation}
then $T^f_\eps\ge T$ and $\sup_{t<T}\{\norm{u(t)}_1+(1+t)^{q/2}\norm{u(t)}_\infty\}\le A\eps$.
\end{theorem}

\begin{proof}
Let $A=2C_{\rm L}\mathcal U_0$, $B=A\eps$, $m(s)=B(1+s)^{-q/2}$ and
$g(s)=(1+s)^{q/2}f(m(s))$. If $\norm w_{\mathcal X_T}\le B$ then, as
$f(\xi)/\xi$ is nondecreasing, $f(w)\le w\,f(m)/m$, so
$\norm{f(w(s))}_1\le g(s)$ and $\norm{f(w(s))}_\infty\le(1+s)^{-q/2}g(s)$. These
bounds replace \eqref{eq:up-bounds}, and the computation of
Subsection~\ref{subsec:lower} gives
$\norm{\mathcal D(w)}_{\mathcal X_T}\le C_{\rm N}\int_0^Thg\dd s$ for
$\mathcal D(w)=\int_0^th\,P_{t-s}f(w)\dd s$. With $\kappa=A/(2C_{\rm N})$ the
monotone iteration $w_0=0$, $w_{k+1}=\eps P_tu_0+\mathcal D(w_k)$ stays in the
ball of radius $B$ and increases to a mild solution.
\end{proof}

For $f=s^p$, \eqref{eq:Hf-cond} is Theorem~\ref{thm:master}-(a). It should be
compared with the integral condition on $f$ in
\cite[Thm.~4.2]{ChatzakouKassymovRuzhansky2026}.
If $\lim_{\lambda\to0}\sup_\sigma
f(\lambda\sigma)/(\lambda f(\sigma))=0$ and
$\int_0^\infty h(1+s)^{q/2}f(\nu_0(1+s)^{-q/2})\dd s<\infty$ for some $\nu_0>0$,
then the left side of \eqref{eq:Hf-cond} is $o(\eps)$ at $T=\infty$ \textup{(}take
$\lambda=A\eps/\nu_0$ and $\sigma=\nu_0(1+s)^{-q/2}$\textup{)}, so small data
are global.

\begin{example}\label{ex:logf}
Let $h\equiv1$, $p=p_F$ (so $\alpha=1$) and $f(s)=s^p[\log(e+1/s)]^{-\gamma}$,
$\gamma>0$, which satisfies (F) but is not comparable to a power at $0$. With
$\ell=\log\frac1{A\eps}$ the left side of \eqref{eq:Hf-cond} is at most
$(A\eps)^p\int_0^{\log(1+T)}(\ell+\frac q2y)^{-\gamma}\dd y$. Hence, as
$\eps\to0$: for $\gamma>1$ small data are global (while $\gamma=0$ has finite
lifespan); for $\gamma=1$, $\log\log T_\eps\ge c\,\eps^{-(p-1)}$; for
$0<\gamma<1$, $\log T_\eps\ge c\,\eps^{-(p-1)/(1-\gamma)}$. Only lower bounds are
proved: this $f$ satisfies \eqref{eq:sandwich} for no exponent, so
Theorem~\ref{thm:sandwich} does not apply.
\end{example}

\subsection{A single-window upper bound for convex \texorpdfstring{$f$}{f}}

Let $J(s)=\int_s^\infty d\sigma/f(\sigma)$, which is finite exactly under
Osgood's condition; $J(y_0)$ is the blow-up time of the ordinary differential
equation $y'=f(y)$, $y(0)=y_0$. For continuous $h>0$ and continuous data, the
next proposition follows from the single-time criterion
\cite[Thm.~3.3]{MeglioliPunzo2026} (see also \cite{LenzSchmidtZimmermann2023}
for $h\equiv1$), $\norm{P_Tu_0}_\infty\ge J^{-1}\bigl(\int_0^Th\bigr)$, combined
with the lower Gaussian bound for $P_Tu_0$; we include the short proof since it
allows $h\in L^1_{\rm loc}$.

\begin{proposition}\label{prop:window}
Let $f$ be convex, $f(0)=0$, $f>0$ on $(0,\infty)$, $\int^\infty ds/f<\infty$.
There is $c_1>0$ such that for all $T,\eps>0$
\[
 J\bigl(c_1\eps(1+T)^{-q/2}\bigr)\le\int_0^Th\dd t\ \Longrightarrow\ T^f_\eps\le T .
\]
\end{proposition}

\begin{inwords}
The solution blows up before time $T$ as soon as the total amount of reaction
$\int_0^Th$ exceeds the ODE blow-up time starting from the height
$c_1\eps(1+T)^{-q/2}$ of the linear flow at time $T$.
\end{inwords}

\begin{proof}
If $T^f_\eps>T$, take $R=\max\{R_0,\sqrt{2T/\omega}\}\le C_0(1+T)^{1/2}$ and
$G(t)=P_{\omega R^2-t}u(t)(0)$. As in Step~1 of Subsection~\ref{subsec:upper},
$G'=h\int\Gamma_{\omega R^2-t}(0,y)f(u)\dd y\ge h\,f(G)$ by Jensen's inequality,
and $G(0)\ge c_{\rm s}\eps R^{-q}>0$. Integrating $(-J(G))'\ge h$ gives
$J(G(0))>\int_0^Th$, and $J$ is decreasing with
$c_{\rm s}R^{-q}\ge c_1(1+T)^{-q/2}$, $c_1=c_{\rm s}C_0^{-q}$.
\end{proof}

For $f=s^p$ this reads $\int_0^Th\ge c\,\eps^{-(p-1)}(1+T)^\alpha$, which is
sharp when the weighted integral is dominated by late times (e.g.\ $h\equiv1$,
$p<p_F$), but says nothing at the critical power, where the nonlinear effect
accumulates over all dyadic scales.

\section*{Acknowledgments}
MR was supported by the Methusalem programme of the Ghent University Special Research Fund (BOF) (Grant number 01M01021), by the FWO Senior Research Grant G083525N and the UKRI EPSRC grant UKRI3645.
BT was supported by the Science Committee of the Ministry of Science and Higher Education of the Republic of Kazakhstan (Grant No. BR31714735).

\section*{Declaration on the use of generative AI.}
The authors used Claude (Anthropic) for language editing, manuscript presentation, and checking mathematical arguments. All AI-assisted content was independently verified, and the authors take full responsibility for the manuscript.

\section*{Declaration of competing interest}
The authors declare that there is no conflict of interest.

\section*{Data Availability Statements}
The manuscript has no associated data.

\end{document}